\documentclass[11pt,a4paper]{article}

\usepackage{epsf,epsfig,amsfonts,amsgen,amsmath,amstext,amsbsy,amsopn,amsthm
}
\usepackage{amsmath,times,mathptmx}
\usepackage{enumitem}
\usepackage{amsfonts,amsthm,amssymb}
\usepackage{amsfonts}
\usepackage{graphics}
\usepackage{latexsym,bm}
\usepackage{amsfonts,amsthm,amssymb,bbding}
\usepackage{indentfirst}
\usepackage{graphicx}
\usepackage{color}
\usepackage[colorlinks=true,anchorcolor=blue,filecolor=blue,linkcolor=blue,urlcolor=blue,citecolor=blue]{hyperref}
\usepackage{float}
\usepackage{tikz}

\DeclareMathAlphabet{\mathcal}{OMS}{cmsy}{m}{n}
\DeclareSymbolFont{largesymbols}{OMX}{cmex}{m}{n}

\newtheorem{thm}{Theorem}

\newtheorem{lemma}{Lemma}
\newtheorem{false statement}{False statement}

\theoremstyle{definition}
\newtheorem{definition}{Definition}

\newtheorem{claimmm}{Claim}
\newtheorem{claimmmm}{Claim}

\newtheorem{claiim}{Claim}
\newtheorem{claiiim}{Claim}

\newtheorem{remark}
{Remark}

\newtheorem{case}{Case}
\newtheorem{casee}{Case}
\newtheorem{caseee}{Case}
\newtheorem{subcase}{Subcase}[case]

\newtheorem{problem}{Problem}

\begin{document}

\title{Planar and Outerplanar Spectral Extremal Problems based on Paths
\footnote{Supported by Natural Science Foundation of Xinjiang Uygur Autonomous Region (No. 2024D01C41) and NSFC (No. 12361071).}}
\author{{Xilong Yin, Dan Li\thanks{Corresponding author. E-mail: ldxjedu@163.com (D. Li).}, Jixiang Meng}\\ 
{\footnotesize College of Mathematics and System Science, Xinjiang University, Urumqi 830046, China}}
\date{}

\maketitle {\flushleft\large\bf Abstract:}
Let SPEX$_\mathcal{P}(n,F)$ and SPEX$_\mathcal{OP}(n,F)$ denote the sets of graphs with the maximum spectral radius over all $n$-vertex $F$-free planar and outerplanar graphs, respectively. Define $tP_l$ as a linear forest of $t$ vertex-disjoint $l$-paths and $P_{t\cdot l}$ as a starlike tree with $t$ branches of length $l-1$. Building on the structural framework by Tait and Tobin [J. Combin. Theory Ser. B, 2017] and the works of Fang, Lin and Shi [J. Graph Theory, 2024] on the planar spectral extremal graphs without vertex-disjoint cycles, this paper determines the extremal graphs in $\text{SPEX}_\mathcal{P}(n,tP_l)$ and $\text{SPEX}_\mathcal{OP}(n,tP_l)$ for sufficiently large $n$. When $t=1$, since $tP_l$ is a path of a specific length, our results adapt Nikiforov's findings [Linear Algebra Appl. 2010] under the (outer)planarity condition. When $l=2$, note that $tP_l$ consists of $t$ independent $K_2$, then as a corollary, we generalize the results of Wang, Huang and Lin [arXiv: 2402.16419] and Yin and Li [arXiv:2409.18598v2]. Moreover, motivated by results of Zhai and Liu [Adv. in Appl. Math, 2024], we consider the extremal problems for edge-disjoint paths and determine the extremal graphs in $\text{SPEX}_\mathcal{P}(n,P_{t\cdot l})$ and $\text{SPEX}_\mathcal{OP}(n,P_{t\cdot l})$ for sufficiently large $n$.

\vspace{0.1cm}
\begin{flushleft}
\end{flushleft}

\section{Introduction}
Planar and outerplanar graphs are fundamental objects in graph theory, with their spectral radius capturing essential structural properties.
A graph is planar if it can be embedded in the plane, that is, it can be drawn on the plane in such a way that edges intersect only at their endpoints. A graph is outerplanar if it can be embedded in the plane so that all vertices lie on the boundary of its outer face.
Denote \( G = (V(G), E(G)) \) be a graph with vertex set \( V(G) \) and edge set \( E(G) \). For \( v \in V(G) \).
The \emph{adjacency matrix} \( A(G) \) is defined as \( A_{uv} = 1 \) if \( uv \in E(G) \), and \( 0 \) otherwise. The \emph{spectral radius} \( \rho(G) \) is the largest eigenvalue of \( A(G) \).
Maximizing the spectral radius under planarity constraints has attracted significant attention, due to its combinatorial implications and applications in network science and geography.

Research on spectral extremal problems in planar graphs traces back to Schwenk and Wilson's provocative question: "What can be said about the eigenvalues of a planar graph?" In 1988, Yuan \cite{H. Yuan} established the first upper bound $\rho(G) \leq 5n - 11$ for $n$-vertex planar graph $G$, and this bound was subsequently improved in a series of papers \cite{M.N. Ellingham}, \cite{B. D. Guiduli}, \cite{H. Yuan-2} and \cite{H. Yuan-3}.
On the other hand, Boots and Royle \cite{B. N. Boots}, and independently, Cao and Vince \cite{D. Cao}, conjectured that $K_2 \vee P_{n-2}$ attains the maximum spectral radius among all planar graphs on $n\geq9$ vertices.
In 2017, Tait and Tobin \cite{M. Tait} confirmed the conjecture for sufficiently large $n$. Up to now, for small $n$, this conjecture is still open.
In fact, even earlier than the Boots–Royle–Cao–Vince conjecture, Cvetković and Rowlinson \cite{D. Cvetković} had considered the spectral extremal problem in outerplanar graphs and proposed a conjecture on the maximum spectral radius of outerplanar graphs in 1990.
Tait and Tobin \cite{M. Tait} also confirmed the Cvetković-Rowlinson conjecture for sufficiently large $n$. Recently, Lin and Ning \cite{H.Q. Lin} completely resolved this conjecture.

Given a graph family $\mathcal{F}$, a graph is said to be $\mathcal{F}$-free if it does not contain any $F\in\mathcal{F}$ as a subgraph.
When $\mathcal{F}=\{F\}$, we write $F$-free.
In 2010, Nikiforov \cite{V. Nikiforov-4} proposed a spectral vertion of Turán-type problem 
: Given a graph $F$, what is the maximum spectral radius of an $F$-free graph of order $n$? Specifically, Nikiforov first showed that the spectral extremal $P_{2k+2}$-free and $P_{2k+3}$-free graphs are $S_{n,k}(= K_k \vee K_{n-k}$, the join of $K_k$ and $K_{n-k}$) and $S^{+}_{n,k}$(the graph obtained by adding one edge within the independent set of $S_{n,k}$), respectively.
Currently, the study of spectral Turán problems has been extended to various subgraphs $F$, such as complete graphs $K_r$ \cite{B. Bollobás,H. Wilf}, independent edges $M_{k}$ \cite{L.H. Feng}, paths $P_{k}$ \cite{V. Nikiforov-4}, friendship graphs $F_{k}$ \cite{V. Nikiforov-3}, cycles \cite{V. Nikiforov-1,M.Q. Zhai-1,M.Q. Zhai-2}, wheels $W_{k}$ \cite{Y.H. Zhao,S. Cioabă-2},
stars $\bigcup _{i= 1}^{k}S_{a_{i}}$ \cite{M.Z. Chen-1} and linear forests $\bigcup _{i= 1}^{k}P_{a_{i}}$ \cite{M.Z. Chen-2}. For more about this topic, readers are referred to three surveys \cite{Y.T. Li,V. Nikiforov-3}.
In 2024, Cioabă, Desai and Tait \cite{S. Cioabă-1} confırmed a famous conjecture of Nikiforov \cite{V. Nikiforov-4} about charactering the extremal graphs in SPEX$(n,C_{2k})$ for any $k\geq3$ and sufficiently large $n$. Fang, Lin, Shu and Zhang \cite{L.F. Fang-1} characterized the extremal graphs with $\emph{spex}(n,F)$ for various specific trees $F$.
While these results apply to general graphs, the constraints of planarity pose new challenges.
Let \(P_l\) be the path with \(l\) vertices. A linear forest \(F(l_1,l_2,\ldots l_t)\) consists of $t$ vertex-disjoint paths \(P_{l_1}, P_{l_2}, \dots, P_{l_k}\) such that \(F(l_1,l_2,\ldots l_t) =\bigcup^{t}_{i=1}P_{l_i}\), where $l_1 \geq l_2 \geq \cdots \geq l_t$.
Then we denote \(S(l_1, l_2, \dots, l_k)\) a starlike tree in which removing the central vertex \(v_1\) leaves a linear forest such that \(S(l_1, l_2, \dots, l_k) - v_1 = \bigcup^{k}_{i=1}P_{l_i}\). We say that the starlike tree \(S(l_1, l_2, \dots, l_k)\) has \(k\) branches and the lengths of branches are \(l_1 \geq l_2 \geq \cdots \geq l_k\).
In this paper, we initiate the study of spectral Turán problems based on paths in outerplanar and planar graphs.

\subsection{The spectral Turán problems for outerplanar graphs}
An $F$‐free outerplanar graph on $n$ vertices with maximum spectral radius is called an extremal graph of $\emph{spex}_\mathcal{OP}(n,F)$ and SPEX$_\mathcal{OP}(n,F)$ represents the family of $F$-free graphs of order $n$ with spectral radius equal to $\emph{spex}_\mathcal{OP}(n,F)$.
Recently, Yin and Li \cite{X.L. Yin} provided the structural results for special $F$-free outerplanar graphs and the graph transformations that increase spectral radius, and then characterized the extremal graphs in SPEX$_\mathcal{OP}(n,B_{tl})$ and SPEX$_\mathcal{OP}(n,(k+1)K_2)$ for all $t \geq 1$, $l \geq 3$, $k \geq 1$ and sufficiently large $n$, where $B_{tl}$ is the graph obtained by sharing a common vertex among $t$ edge-disjoint $l$-cycles and $(t+1)K_{2}$ is the disjoint union of $t+1$ copies of $K_2$.
Recall $tP_l$ is a linear forest of $t$ vertex-disjoint $l$-paths, $P_{t\cdot l}$ is a starlike tree with $t$ branches of length $l-1$, and $\mathcal{G}_{\mathcal{OP}}(n, F)$ denote the class of $n$-vertex $F$-free outerplanar graphs. We are concerned with the following problem:
\begin{problem}\label{p1}
What is the maximum spectral radius $\rho(G)$ over all $G \in \mathcal{G}_{\mathcal{OP}}(n, P_{t\cdot l})$ or $\mathcal{G}_{\mathcal{OP}}(n, tP_l)$?
\end{problem}

\begin{remark}\label{rk1}
Problem \ref{p1} is equivalent to determining which $P_{t\cdot l}$-free (or $tP_{l}$-free) outerplanar graphs are the extremal graphs to $\emph{spex}_{\mathcal{OP}}(n,P_{t\cdot l})$ (or $\emph{spex}_{\mathcal{OP}}(n,tP_{l})$). For $t=1$, Nikiforov \cite{V. Nikiforov-4} determined that the spectral extremal graph among $P_{2k+2}$‐free and $P_{2k+3}$‐free graphs is $S_{n,k} \cong K_{k} \vee \overline{K_{n-k}}$ and $S^{+}_{n,k}$ for $k \geq 1$, respectively. In addition, $S_{n,k}$ and $S^{+}_{n,k}$ are outerplanar graphs for $k=1$ obviously. This implies that we already have the answer about the extremal graph to $\emph{spex}_{\mathcal{OP}}(n,P_{l})$ for $l = 4,5$ and the case for $l=2,3$ are trivial.
\end{remark}
\begin{definition}
For three positive integers $n,n_1,n_2$ with $n\geq n_1>n_2 \geq 1$. Define $H(n_1, n_2)$ as follows:
\[  
H_{\mathcal{OP}}(n_1, n_2) = \begin{cases}   
P_{n_1} \cup \frac{n-1-n_1}{n_2} P_{n_2}, & \text{if } n_2|(n-1-n_1); \\  
P_{n_1} \cup \left\lfloor \frac{n-1-n_1}{n_2} \right\rfloor P_{n_2} \cup P_{n-1-n_1-\left\lfloor \frac{n-1-n_1}{n_2} \right\rfloor n_2}, & \text{otherwise.}  
\end{cases}  
\]
\end{definition}

The proofs of the following Theorems \ref{thm1} and \ref{thm2} can provide a solution Problem \ref{thm1}.

\begin{thm}\label{thm1}
For $n_0 = max \{1.27\times 10^7, 6.5025\times2^{(t-1)(l-1)+l}\}$, 
 \begin{description}
   \item[(i)] When $t=1$, $l \geq 4 $ and $n \geq max\{n_0, (5.08\left\lfloor\frac{l-2}2\right\rfloor)^2+1\}$, the extremal graph to $\textit{spex}_{\mathcal{OP}}(n,P_{t\cdot l})$ is $K_1 \vee H_{\mathcal{OP}}(\lceil\frac{l-2}{2}\rceil,\lfloor\frac{l-2}{2}\rfloor)$;
   \item[(ii)] When $t=2$, $l \geq 3 $ and $n \geq max\{n_0, (5.08\left\lfloor\frac{2l-3}2\right\rfloor)^2+1\}$, the extremal graph to $\textit{spex}_{\mathcal{OP}}(n,P_{t\cdot l})$ is $K_1 \vee H_{\mathcal{OP}}(\lceil \frac{2l-3}{2} \rceil,\lfloor \frac{2l-3}{2} \rfloor)$;
   \item[(iii)] When $t\geq 4$, $l \geq 3 $ and $n \geq n_0$, the extremal graph to $\textit{spex}_{\mathcal{OP}}(n,P_{t\cdot l})$ is $K_1 \vee H_{\mathcal{OP}}(tl-t-1,l-2)$;
 \end{description}
\end{thm}

\begin{thm}\label{thm2}
For
$n \geq N+\frac{3}{2}+3\sqrt{N-\frac{7}{4}}$, where $N = max\{1.27\times 10^7, 6.5025\times2^{tl}, (5.08\left\lfloor\frac{l-2}2\right\rfloor)^2+1\}$,
\begin{description}
   \item[(i)]When $t=1$ and $l \geq 4$, the extremal graph to $\textit{spex}_{\mathcal{OP}}(n,tP_l)$ is $K_1 \vee H_{\mathcal{OP}}(\lceil\frac{l-2}{2}\rceil,\lfloor\frac{l-2}{2}\rfloor)$;
   \item[(i)]When $t \geq 2$ and $l \geq 3$, the extremal graph to $\textit{spex}_{\mathcal{OP}}(n,tP_l)$ is $K_1 \vee H_{\mathcal{OP}}(tl-l-1,l-1)$.
\end{description}
\end{thm}

\subsection{The spectral Turán problems for planar graphs}
Furthermore, we aim to extend the above results to the planar spectral Tur$\mathrm{\acute{a}}$n problems for paths.
An $F$‐free planar graph on $n$ vertices with maximum spectral radius is called an extremal graph of spex$_\mathcal{P}(n,F)$ and SPEX$_\mathcal{P}(n,F)$ represents the family of $F$-free graphs of order $n$ with spectral radius equal to $\emph{spex}_\mathcal{P}(n,F)$.
In 2022, Zhai and Liu \cite{M.Q. Zhai-2} determined the extremal graphs in SPEX$_\mathcal{P}(n,\mathcal{F})$ when $\mathcal{F}$ is the family of $k$ edge-disjoint cycles.
In 2023, Fang, Lin and Shi \cite{L.F. Fang-2} characterized the extremal graphs in SPEX$_\mathcal{P}(n,tC_\ell)$ and SPEX$_\mathcal{P}(n,tC)$ by providing the structural theorem and the graph transformations that increase the spectral radius, where $tC_\ell$ is $t$ vertex-disjoint $\ell$-cycles, and $tC$ is the family of $t$ vertex-disjoint cycles without length restriction.
As application, Zhang and Wang \cite{H.R. Zhang} characterized the unique extremal planar graph with the maximum spectral radius among $C_{l,l}$-free planar graphs and $\theta_k$-free planar graphs on $n$ vertices, where $C_{l,l}$ be a graph obtained from $2C_l$ such that the two cycles share a common vertex and $\theta_k$ are the Theta graph of order $k \geq 4.$
Similarly, Wang, Huang and Lin \cite{X.L. Wang} optimized the above structural theorem and graph transformation, which determining the extremal graphs in $\operatorname{SPEX}_{\mathcal{P}}(n,W_{k}),\operatorname{SPEX}_{\mathcal{P}}(n,F_{k})$ and SPEX$_\mathcal{P}(n,(k+1)K_{2})$ when $W_{k},F_{k}$ and $(k+1)K_{2}$ are the wheel graph of order $k$, the friendship graph of order $2k+1$ and the disjoint union of $k+1$ copies of $K_2$, respectively.
Recently, Yin and Li \cite{X.L. Yin} further determined the extremal graphs in SPEX$_\mathcal{P}(n,B_{tl})$ for all $t \geq 1$, $l \geq 3$ and sufficiently large $n$, where $B_{tl}$ is the graph obtained by sharing a common vertex among $t$ edge-disjoint $l$-cycles.
Recall $tP_l$ is a linear forest of $t$ vertex-disjoint $l$-paths, $P_{t\cdot l}$ is a starlike tree with $t$ branches of length $l-1$, and let $\mathcal{G}_{\mathcal{P}}(n, F)$ denote the class of $n$-vertex $F$-free planar graphs. Naturally, we consider the planar version of Problem \ref{p1}:
\begin{problem}\label{p2}
What is the maximum spectral radius $\rho(G)$ over all $G \in \mathcal{G}_{\mathcal{P}}(n, P_{t\cdot l})$ or $\mathcal{G}_{\mathcal{P}}(n, tP_l)$?
\end{problem}


\begin{remark}\label{rk2}
Similar to Remark \ref{rk1}, Problem \ref{p2} is equivalent to determining which $P_{t\cdot l}$-free (or $tP_{l}$-free) planar graphs are the extremal graphs to $\emph{spex}_{\mathcal{OP}}(n,P_{t\cdot l})$ (or $\emph{spex}_{\mathcal{P}}(n,tP_{l})$). Note that $S_{n,k}$ and $S^{+}_{n,k}$ are planar graphs for $k=1,2$. This implies that we already have the answer about the extremal graph to $\emph{spex}_{\mathcal{P}}(n,P_{l})$ for $l = 4,5,6,7$ and the cases for $l=2,3$ are trivial.
\end{remark}

\begin{definition}
For three positive integers $n,n_1,n_2$ with $n\geq n_1>n_2 \geq 1$. Define $H(n_1, n_2)$ as follows:
\[  
H_{\mathcal{P}}(n_1, n_2) = \begin{cases}   
P_{n_1} \cup \frac{n-2-n_1}{n_2} P_{n_2}, & \text{if } n_2|(n-2-n_1); \\  
P_{n_1} \cup \left\lfloor \frac{n-2-n_1}{n_2} \right\rfloor P_{n_2} \cup P_{n-2-n_1-\left\lfloor \frac{n-2-n_1}{n_2} \right\rfloor n_2}, & \text{otherwise.}  
\end{cases}  
\]
\end{definition}

\begin{definition}
For four positive integers $n,n_1,n_2,n_3$ with $n\geq n_1>n_2>n_3 \geq 1$. Define $H(n_1, n_2, n_3)$ as follows:
\[  
H_{\mathcal{P}}(n_1, n_2, n_3) = \begin{cases}   
P_{n_1} \cup P_{n_2} \cup \frac{n-2-n_1-n_2}{n_3} P_{n_3}, & \text{if } n_3|(n-2-n_1-n_2); \\  
P_{n_1} \cup P_{n_2} \cup \left\lfloor \frac{n-2-n_1-n_3}{n_3} \right\rfloor P_{n_3} \cup P_{n-2-n_1-n_2-\left\lfloor \frac{n-2-n_1-n_2}{n_3} \right\rfloor n_3}, & \text{otherwise.}  
\end{cases}  
\]
\end{definition}

Our proofs of Theorems \ref{thm3} and \ref{thm4} provide a solution to Problem \ref{p2}.

\begin{thm}\label{thm3}
For $n_0 = max\{2.67\times 9^{17}, 10.2\times2^{(t-1)(l-1)+l-3}+2\}$,
 \begin{description}
   \item[(i)]When $t=1$, $l \geq 6$ and $n \geq max \{n_0, \frac{625}{32}{\lfloor \frac{l-3}{3} \rfloor}^2 +2 \}$, the extremal graph to $\textit{spex}_{\mathcal{P}}(n,P_{t\cdot l})$ is $K_2 \vee H_{\mathcal{P}}(l-3-2\lfloor\frac{l-3}{3}\rfloor,\lfloor\frac{l-3}{3}\rfloor)$;
   \item[(ii)]When $t=2$, $l \geq 4$ and $n \geq max \{n_0, \frac{625}{32}{\lfloor \frac{2l-4}{3} \rfloor}^2 +2 \}$, the extremal graph to $\textit{spex}_{\mathcal{P}}(n,P_{t\cdot l})$ is $K_2 \vee H_{\mathcal{P}}(2l-4-2\lfloor\frac{2l-4}{3}\rfloor,\lfloor\frac{2l-4}{3}\rfloor)$;
   \item[(iii)]When $t \geq 5$, $l \geq 3$ and $n \geq max \{n_0, \frac{625}{32}{\lfloor \frac{l-3}{2} \rfloor}^2 +2 \}$, the extremal graph to $\textit{spex}_{\mathcal{P}}(n,P_{t\cdot l})$ is $K_2 \vee H_{\mathcal{P}}(tl-t-l,l-2)$;
 \end{description}
\end{thm}

\begin{thm}\label{thm4}
For $N_0 = max\{2.67 \times 9^{17}, 10.2 \times 2^{tl-3}+2\}$
 \begin{description}
   \item[(i)]When $t=1$, $l \geq 6$, $N= max \{N_0, \frac{625}{32}{\lfloor \frac{l-3}{3} \rfloor}^2 +2\}$, and $n \geq N+\frac{3}{2}\sqrt{2N-6}$, the extremal graph to $\textit{spex}_{\mathcal{P}}(n,tP_l)$ is $K_2 \vee H_{\mathcal{P}}(l-3-2\lfloor\frac{l-3}{3}\rfloor,\lfloor\frac{l-3}{3}\rfloor)$;
   \item[(ii)]When $t=2$, $l \geq 4$, $N= max \{N_0, \frac{625}{32}{\lfloor \frac{l-2}{2} \rfloor}^2 +2\}$, and $n \geq N+\frac{3}{2}\sqrt{2N-6}$, the extremal graph to $\textit{spex}_{\mathcal{P}}(n,tP_{l})$ is $K_2 \vee H_{\mathcal{P}}(l-1, \lceil \frac{l-2}{2}\rceil, \lfloor \frac{l-2}{2}\rfloor)$;
   \item[(iii)]When $t \geq 3$, $l \geq 3$, $N= max \{N_0, \frac{625}{32}({\lfloor \frac{l-2}{2} \rfloor+1})^2 +2,
        \}$, and $n \geq N+\frac{3}{2}\sqrt{2N-6}$, the extremal graph to $\textit{spex}_{\mathcal{P}}(n,tP_{l})$ is $K_2 \vee H_{\mathcal{P}}(tl-2l-1,l-1)$;
 \end{description}
\end{thm}

\vspace*{2mm}
The paper is organized as follows.
The Section \ref{sc3} provides the proofs of Theorems \ref{thm1} and \ref{thm2}, respectively.
The Section \ref{sc4} is dedicated to the proofs of Theorems \ref{thm3} and \ref{thm4}, respectively.


\section{Proofs of Theorems \ref{thm1} and \ref{thm2}}\label{sc3}

In this section, we provide the detailed proofs of Theorems \ref{thm1} and \ref{thm2}.
We first introduce some useful and classic spectral tools as follow.
Suppose $A$ is a $n\times n$ real symmetric matrix (or Hermitian matrix in general). Then the eigenvalues of $A$ are all real, and we may label them in non-increasing order as
$$\lambda_1\geq\lambda_2\geq\cdots\geq\lambda_n.$$
Recall that a sequence $\mu_1\geq\cdots\geq\mu_m$ is said to interlace a sequence $\lambda_1\geq\cdots\geq\lambda_n$ with $m<n$ when $\lambda_i\geq\mu_i\geq\lambda_{n-m+i}$ for $i=1,\ldots,m.$ A corollary of Cauchy's Interlacing Theorem states that if $B$ is a principal submatrix of a symmetric matrix $A$, then the eigenvalues of $B$ interlace the eigenvalues of $A[5].$ In particular, the eigenvalues of a proper induced subgraph $H$ of $G$ interlace the eigenvalues of $G.$
Let $\boldsymbol x=(x_1,x_2,\cdots,x_n)^{\mathrm{T}}$ be an eigenvector of $A(G)$ corresponding to $\rho(G).$ The eigenequation for any vertex $v$ of $G$ is
$$\rho(G)x_v=\sum_{uv\in E(G)}x_u.$$
The Rayleigh quotient of $G$ is expressed as:
$$\rho(G)=\max_{\boldsymbol{x}\in\mathbb{R}^n}\frac{\boldsymbol{x}^\mathrm{T}A(G)\boldsymbol{x}}{\boldsymbol{x}^\mathrm{T}\boldsymbol{x}}=\max_{\boldsymbol{x}\in\mathbb{R}^n}\frac{2\sum_{uv\in E(G)}x_ux_v}{\boldsymbol{x}^\mathrm{T}\boldsymbol{x}},$$
where $\mathbb{R}^n$ is the set of $n$-dimensional real numbers.
If $G$ is a connected graph with order $n$, by Perron-Frobenius theorem, then there exists a positive eigenvector $X=(x_1,...,x_n)^T$ corresponding to $\rho(G)$.

We then present a series of key structural lemmas for outerplanar graphs, which provide crucial insights into the properties and characteristics of extremal graphs.

\begin{lemma}[Key Structure of Outerplanar Extremal Graphs]\cite{X.L. Yin}\label{lm1}
 Let $F$ be an outerplanar graph contained in $K_{1} \vee P_{n-1}$ but not contained in $K_{1} \vee ((t-1)K_2\cup(n-2t+1)K_1)$,
where $1 \leq t \leq \frac{n-1}{2} $ and $n\geq max\{1.27\times 10^7,[(64+|V(F)|)^\frac{1}{2}+8]^2\}$. Suppose that $G$ is a connected extremal graph with $\textit{spex}_{\mathcal{OP}}(n,F)$ and $X$ 
is the positive eigenvector of $\rho(G)$ with $\emph{max}_{v\in V(G)}x_v = 1 $.
Then there exists a vertice $u \in V(G)$
   such that $N_G(u)=V(G)\setminus\{u\}$ and $x_{u}=1$, in especially,

 \begin{description}
   \item[(i)]  $G$ contains a copy of $K_{1,n-1}.$
   \item[(ii)] The subgraph $G[N_G(u)]$ is a disjoint union of some paths.
 \end{description}
\end{lemma}

\begin{lemma}[Eigenvector Bounds]\cite{X.L. Yin}\label{lm2}
Suppose further that G contains $K_{1, n-1}$ as a subgraph. Let $u_{1}$ be the two vertices of $G$ that have degree $n-1$ in $K _{1, n-1}$.
For any vertex $u\in V( G) \setminus \{ u_1\} $, we have 
$$\frac 1\rho \leq x_u\leq \frac 1\rho + \frac {2.04}{\rho ^2}.$$
\end{lemma}

\begin{definition}[Outerplanar Path Transformation]\cite{X.L. Yin}\label{de1}
Let $s_1$ and $s_2$ be two integers with $s_1\geq s_2\geq1$, and let $H=P_{s_1}\cup P_{s_2}\cup H_0$, where $H_0$ is a disjoint union of paths. We say that $H^*$ is an $(s_1,s_2)$- transformation of $H$ if
$$H^*:=\begin{cases}P_{s_1+1}\cup P_{s_2-1}\cup H_0&\text{if }s_2\geq2,\\P_{s_1+s_2}\cup H_0&\text{if }s_2=1.\end{cases}$$
Clearly, $H^*$ is a disjoint union of paths, which implies that $K_1 \vee H^*$ is an outerplanar graph.
\end{definition}

This transformation strictly increases the spectral radius of \( K_1 \vee H \) for large \( n \).

\begin{lemma}[Spectral Comparison for Outerplanar Transformations]\cite{X.L. Yin}\label{lm3}
  For \( H \) and \( H^* \) as in Definition \ref{de1}, if \( n \geq max\{1.27\times 10^7,6.5025\times2^{s_2+2}\} \), then \( \rho(K_1 \vee H^*) > \rho(K_1 \vee H) \).
\end{lemma}
\vspace*{1mm}
\begin{proof}[\textbf{Proof of Theorem \ref{thm1}}]

Let $G$ be the extremal graph with $\textit{spex}_{\mathcal{OP}}(n,P_{t\cdot l})$.
Noting that $F\in\{P_{t\cdot l}|t = 1, l \geq 6\}$ or $F\in\{P_{t\cdot l}|t \geq 2, l \geq 3\}$ is a subgraph of $K_{1} \vee P_{n-1}$ and is not of $K_{1} \vee ((t-1)K_2\cup(n-2t+1)K_1)$ for large enough $n$,
then $G\cong K_{1} \vee G[A]$, where $G[A] \cong G[N_G(u)]$ is a disjoint union of paths by Lemma \ref{lm1}.
Suppose that $G[A]=\cup_{i=1}^qP_{n_i}$, where $q\geq2$ and $n_1\geq n_2\geq\cdots\geq n_q$.
Let $H$ be a disjoint union of $q$ paths. We use $n_i(H)$ to denote the order of the $i$-th longest path of $H$ for any $i\in\{1,...,q\}.$ 

(i) When $t=1$, we already know that the outerplanar spectral extremal graph without subgraph $P_l$ for $l \leq 5$ by Remark \ref{rk1}. So, we focus on the case $l\geq 6$.

\begin{claimmm}\label{claim3.1}  
If $H$ is a disjoint union of paths, then $K_1 \vee H$ is $P_{l}$-free if and only if $n_1(H) +n_2(H) \leq l-2$.  
\end{claimmm}  
  
\begin{proof}  
One can observe that the length of the longest path in $K_1 \vee H$ is $n_1(H) +n_2(H) + 1$, and $K_1 \vee (P_{n_1(H)} \cup P_{n_2(H)})$ contains a path $P_i$ for each $i \in \{3, \ldots, n_1(H) +n_2(H) + 1\}$.  
Thus, $n_1(H) +n_2(H) + 1 \leq l - 1$ if and only if $K_1 \vee H$ is $P_l$-free,  
which the claim proves.  
\end{proof}  

By Claim \ref{claim3.1} and direct computation, we have $n_2 \leq n_1 \leq l - 2$ in $G[A]$, and so $6.5025 \times 2^l \geq 6.5025 \times 2^{n_2 + 2}$. Thus,  
\begin{equation}\label{gongshi1}  
n \geq \max\{1.27 \times 10^7, [(64 + |V(F)|)^{1/2} + 8]^2, 6.5025 \times 2^{l}\}.  
\end{equation}  
  
\begin{claimmm}\label{claim3.2}  
$n_1 +n_2 = l - 2$.  
\end{claimmm}  
  
\begin{proof}  
By Claim \ref{claim3.1}, we already know that $n_1+n_2 \leq l-2 $. Now suppose to the contrary that $n_1+n_2\leq\ l-3.$ Let $H^{\prime}$ be an $(n_1,n_t)$-transformation of $G[A].$ Clearly, $n_1(H^{\prime})=n_1+1$ and $n_2(H^{\prime})=n_2.$ Then, $n_1(H^{\prime})+n_2(H^{\prime})\leq l-2.$ By Claim \ref{claim3.1}, $K_1 \vee H^{\prime}$ is $P_l$-free. However, by (\ref{gongshi1}) and Lemma \ref{lm3}, we have $\rho(K_1 \vee H^{\prime})>\rho$, contradicting the maximality of $\rho = \rho(G)$. Hence, the claim holds 
\end{proof}  

\begin{claimmm}\label{claim3.3}  
$n_i = \lfloor \frac{l-2}{2} \rfloor$ for $i \in \{2, \ldots, q - 1\}$.  
\end{claimmm}  
  
\begin{proof}  
Suppose to the contrary, then set $i_0= \min\{i|3\leq i\leq q-1,n_i\leq n_2-1\}.$ Let $H^{\prime\prime}$ be an $(n_{i_0},n_q)$-transformation of $G[A].$ Clearly, $n_1(H^{\prime\prime})=n_1$ and $n_2(H^{\prime\prime})=\max\{n_2,n_{i_0}+1\}=n_2$, and so $n_1(H^{\prime\prime})+n_2(H^{\prime\prime})=l-2.$ By Claim \ref{claim3.3}, $K_2 \vee H^{\prime\prime}$ is $P_l$-free. However, by (\ref{gongshi1}) and Lemma \ref{lm3},  $\rho(K_1 \vee H^{\prime\prime})>\rho$, contradicting the maximality of $\rho = \rho(G)$. It follows that $G=K_1 \vee (P_{n_1}\cup(t-2)P_{n_2}\cup P_{n_t}).$
Finally, we claim that $n_2=\left\lfloor\frac{l-2}2\right\rfloor.$ Note that $1\leq n_2\leq\left\lfloor\frac{l-2}2\right\rfloor.$ Then, $n_2=\left\lfloor\frac{l-2}2\right\rfloor$ for $l\in\{4,5\}.$
It remains the case $l \geq 6.$
Suppose to the contrary, then $n_2\leq\left\lfloor\frac{l-4}2\right\rfloor$ as $n_2<\left\lfloor\frac{l-2}{2}\right\rfloor.$ Using $P^i$ to denote the $i$-th longest path of $G[A]$ for $i\in\{1,...,q\}.$
Note that $n-1 = \sum_{i=1}^{q} n_i \leq n_1 + (q-1)n_2$, it implies that $q \geq n_2 + 4$ by $n \geq 6.5025 \times 2^{l-2+2} \gg (l-2)+2{\lfloor\frac{l-2}{2}\rfloor}+{\lfloor\frac{l-2}{2}\rfloor}^2+ 1$.
Then $P^2,P^3,...,P^{n_2+3}$ are paths of order $n_2.$ We may assume that $P^{n_2+3}=w_1w_2\cdots w_{n_2}.$
Since $n_1=l-2-n_2\geq\left\lceil\frac{l-2}2\right\rceil \geq 2$, there exists an endpoint $w^{\prime\prime}$ of $P^1$ with $w^\prime w^{\prime\prime}\in E\left(P^1\right).$ Let $G^\prime$ be obtained from $G$ by 

(1) deleting $w^\prime w^{\prime\prime}$ and joining $w^{\prime\prime}$ to an endpoint of $P^{n_2+2};$

(2) deleting all edges of $P^{n_2+3}$ and joining $w_i$ to an endpoint of $P^{i+1}$ for each $i\in\{1,...,n_2\}.$ Clearly, $G^\prime$ is obtained from $G$ by deleting $n_2$ edges and adding $n_2+1$ edges. By Lemma \ref{lm2}, we obtain
$$\frac1{\rho^2}\leq x_{u_i}x_{u_j}\leq\frac1{\rho^2}+\frac{4.08}{\rho^3}+\frac{4.1616}{\rho^4}<\frac1{\rho^2}+\frac{5.08}{\rho^3}$$
for any vertices $u_i,u_j\in A.$ Then

$$\rho(G^{\prime})-\rho\geq\frac{X^T(A(G^{\prime})-A(G))X}{X^TX}>\frac2{X^TX}\biggl(\frac{(n_2+1)}{\rho^2}-\frac{n_2}{\rho^2}-\frac{5.08n_2}{\rho^3}\biggr)>0,$$
where $n_2<\left\lfloor\frac{l-2}2\right\rfloor\leq\frac1{5.08}(\sqrt{n-1})\leq\frac1{5.08}\rho$ as $n \geq (5.08\left\lfloor\frac{l-2}2\right\rfloor)^2+1$ and $\rho \geq \rho(K_{1,n-1})=\sqrt{n-1}$ (since $K_{1,n-1}$ is a $P_l$-free outerplanar graph for $l \geq 4$). So, $\rho(G^{\prime})>\rho.$
Meanwhile, $G^\prime\cong K_1 \vee (P_{n_1-1}\cup(n_2+1)P_{n_2+1}\cup(t-n_2-4)P_{n_2}\cup P_{n_t}).$ By Claim \ref{claim3.1}, $G^\prime$ is $P_{l}$-free, contradicting the fact that $G$ is extremal to $spex_{OP}(n,P_l).$ So, the claim holds.
\end{proof}

Since $n_1 + n_2 = l-2$ and $n_2=\lfloor \frac{l-2}{2} \rfloor$, we have $n_1=\lceil \frac{l-2}{2}\rceil$. Moreover, since $n_i=\lfloor \frac{l-2}{2} \rfloor$ for $i \in \{2, \ldots, q - 1\}$ and $n_q \leq n_2$, we can see that $G[A] \cong H_{\mathcal{OP}}(\lceil \frac{l-2}{2} \rceil, \lfloor\frac{l-2}{2}\rfloor)$. This completes the proof of Theorem \ref{thm1} (i).

(ii) If $t = 2$ and $l\geq 3$, then $P_{2\cdot l} \cong P_{2l-1}$. Clearly, by Theorem \ref{thm1} (i), we can directly conclude that $G \cong K_1 \vee H_{\mathcal{OP}}(\lceil l-\frac{3}{2} \rceil,\lfloor l-\frac{3}{2} \rfloor)$, which completes the proof of Theorem \ref{thm1} (ii).

(iii) If $t \geq 4$ and $l\geq 3$, it is not difficult to know that a \(P_{t\cdot l}\)-free graph is also \(B_{tl}\)-free. Meanwhile, as we know from \cite{X.L. Yin}, the spectral extremal \(B_{tl}\)-free outerplanar graph that is $K_1 \vee H_{\mathcal{OP}}(tl-l-1,l-2)$, and $K_1 \vee H_{\mathcal{OP}}(tl-l-1,l-2)$ happens to be \(P_{t\cdot l}\)-free as well. Therefore, it is proven that the extremal graph of \(P_{t\cdot l}\)-free is $K_1 \vee H_{\mathcal{OP}}(tl-l-1,l-2)$ for $t \geq 4$ and $l\geq 3$.
\end{proof}

\vspace*{1mm}
Before our detailed proof of Theorem \ref{thm2}, we need to introduce an upper bound of the spectral radius of outerplanar graphs as following.

\begin{lemma}\cite{J. Shu}.\label{lm7}
Let $G$ be a connected outerplanar graph on $n\geq3$ vertices. Then $\rho\left(G\right)\leq\frac32+\sqrt{n-\frac74}$.
\end{lemma}

\vspace*{1mm}
\begin{proof}[\textbf{Proof of Theorem \ref{thm2}}]

Let $G$ be the extremal graph with $\textit{spex}_{\mathcal{OP}}(n,tP_{l})$. We consider the following two cases.
\begin{case}{$G$ is a connected graph.}\label{case3.1}

Noting that $F\in\{tP_{l}| t = 1 , l \geq 4\}$ or $F\in\{tP_{l}|t \geq 2, l \geq 3\}$ is a subgraph of $K_{1} \vee P_{n-1}$ but not of $K_{1} \vee ((t-1)K_2\cup(n-2t+1)K_1)$ for large enough $n$,
then $G\cong K_{1} \vee G[A]$ where $G[A] \cong G[N_G(u)]$ is a disjoint union of paths by Lemma \ref{lm1}. Suppose that $G[A]=\cup_{i=1}^qP_{n_i}$ with $q\geq2$ and $n_1\geq n_2\geq\cdots\geq n_q$.
Let $H$ be a disjoint union of $q$ paths. We use $n_i(H)$ to denote the order of the $i$-th longest path of $H$ for any $i\in\{1,...,q\}.$ 

\begin{claimmmm}\label{claim3.4}
If $H$ is a union of disjoint paths and $K_1 \vee H$ is $tP_{l}$-free,
then $n_1(H)+n_2(H) \leq (t-2)l+l-1+l-1$.
\end{claimmmm}

\begin{proof}  
Suppose to the contrary that $n_1(H) + n_2(H) \geq (t - 2)l +(l - 1)+ l $.  
Let $a$ and $b$ be the maximum number of independent $l$-paths in $P_{n_1(H)}$ and $P_{n_2(H)}$, respectively, i.e.,  
\[  
a = \left\lfloor \frac{n_1(H)}{l} \right\rfloor \quad \text{and} \quad b = \left\lfloor \frac{n_2(H)}{l} \right\rfloor.  
\]  
This implies that $n_1(H) - al \leq l - 1$ and $n_2(H) - bl \leq l - 1$.
Next, we assert that $a + b = t-1$. Clearly, if $a + b > t-1 $, then $G$ would contain a $tP_{l}$, which is a contradiction.  
If $a + b < t - 1$, recall that $n_1(H) + n_2(H) \geq (t - 2)l +(l - 1)+ l$, then 
\[  
n_1(H) + n_2(H) - al - bl \geq 2l - 1.  
\]  
On the other hand, since $n_1(H) - al \leq l - 1$ and $n_2(H) - bl \leq l - 1$, we have  
\[  
n_1(H) + n_2(H) - al - bl \leq 2l-2 < 2l - 1,  
\]
which is a contradiction. Thus, $a + b = t - 1$. Therefore, $G[P_{n_1(H)} \cup P_{n_2(H)} \cup \{u\}]$ would generate a $tP_{l}$ due to $n_1(H) + n_2(H) \geq (t - 2)l +(l - 1)+ l $, a contradiction.   
\end{proof}

Let $H^* = H_{\mathcal{OP}}((t-2)l+l-1, l-1)$. Clearly, $K_1 \vee H^* = K_1 \vee H_{\mathcal{OP}}((t-2)l+l-1, l-1)$, which means that $n_1 = (t-2)l+l-1$ and $n_2 = l-1$.  
If $G[A] \cong H^*$, then $G \cong K_1 \vee H_{\mathcal{OP}}(tl-l-1, l-1)$, and we are done.  
If $G[A] \ncong H^*$, by Lemma \ref{lm1} and Claim \ref{claim3.4}, then $G[A]$ is a disjoint union of some paths and $n_1+n_2\leq (t-2)l+l-1+l-1$.
Now, consider the quantity $r$ which represents the number of independent $l$-paths in $P_{n_1}$. 
Firstly, we can get $r\leq t-2$ for $l=3$ and $r \leq t-1$ for $l \geq 4$. Otherwise, there would be a $tP_l$ in $G$, a contradiction.
Next, we will follow the following cases.
  
\begin{subcase}\label{subcase2.1}  
$r \leq t-2$ and $l \geq 3$. We can obtain $H^*$ by applying a series of $(s_1, s_2)$-transformations to $G[A]$, where $s_2 \leq s_1 \leq (t-2)l+l-1+l-1$.
As $n \geq 6.5025 \times 2^{(t-2)(l-1)+l-1+l-1}$ $\geq 6.5025 \times 2^{s_2} $, we conclude that $\rho < \rho(K_1 \vee H^*) = \rho(K_1 \vee H_{\mathcal{OP}}(tl-l-1, l-1)$ by Lemma \ref{lm5},  
which is impossible by the maximality of $\rho = \rho(G)$.
Up to this case, we can conclude that $G\cong K_2 \vee H_{\mathcal{OP}}(tl-l-1, l-1)$.
\end{subcase}  

\begin{subcase}  
$r = t-1$ and $l \geq 4$. Then we have the following claims.
\begin{claimmmm}\label{claim3.5}  
Let $P_{n_1} = v_1v_2 \cdots v_{n_1}$, $P' = v_1v_2 \cdots v_{(t-1)l}$ and $P_{\bar{n}_1} = v_{(t-1)l+1} \cdots v_{n_1}$, where $\bar{n}_1=n_1-(t-1)l-1$.
Then $K_1 \vee G[A]$ does not contain $tP_l$ if and only if $\bar{n}_1+n_2 \leq l-2$ or $n_2+n_3 \leq l-2$.
\end{claimmmm}

\begin{proof}
Observing that the longest path which belongs to $K_1 \vee (G[A]-P')$ has an order of $\bar{n}_1 + n_2+1$ or $n_2+n_3+1$.
Furthermore, $K_1 \vee (P_{\bar{n}_1} \cup P_{n_2})$ or $K_1 \vee (P_{n_2} \cup P_{n_3})$ contains a path $P_i$ for each $i \in \{3,\ldots,\max\{\bar{n}_1+ n_2 + 1, n_2+n_3 +1\}\}.$
Therefore, $\bar{n}_1+n_2 \leq l-2$ or $n_2+n_3 \leq l-2$ if and only if $K_1 \vee G[A]$ is $tP_l$‐free.
Hence, the claim holds.
\end{proof}

\begin{claimmmm}\label{claim3.6}  
$n_i=n_2$ for $i \in \{3,\ldots,q-1\}$.
\end{claimmmm}

\begin{proof}
Suppose to the contrary, then there exists an $i_0 = \min\{i|3 \leq i \leq q-1, n_i \leq n_2-1\}$.
If $i_0 > 3$, then let $H'$ be an $(n_{i_0},n_q)$-transformation of $G[A]$. Thus, $n_1(H')=n_1$(it means $\bar{n}_1(H') = \bar{n}_1$), $n_2(H')=\max\{n_2,n_{i_0}+1\}=n_2$ and $n_3(H')=n_3$. Therefore, $\bar{n}_1(H')+n_2(H') \leq l-2$ or $n_2(H')+n_3(H') \leq l-2$. By Claim \ref{claim3.5}, $K_1 \vee H'$ is $tP_l$-free. However, by Lemma \ref{lm3}, we have $\rho(K_1 \vee H') > \rho$, a contradiction.
If $i_0 =3$, we first suppose that there exists a $j$ such that $n_j < n_3 \leq n_2-1$, where $4 \leq j \leq q-1$. Let $H''$ be the $(n_j,n_q)$-transformation of $G[A]$, where $4 \leq j \leq q-1.$ Then $\bar{n}_{1}(H'') = \bar{n}_{1}$, $n_{2}(H'') = n_{2}$ and $n_{3}(H'')=\max\{n_{3},n_{j}+1\}=n_{3}.$ Using the methods similar to those for $i_0>3$, we can also get a contradiction. Then we suppose that $n_j=n_3$ for any $j\in\{4,\cdots,q-1\}$.
If $\bar{n}_1+n_2 \leq l-2$, let $H^{(3)}$ be the $(n_3,n_q)$-transformation of $G[A]$, then $n_{1}(H^{(3)})=n_{1},n_{2}(H^{(3)})=\max\{n_{2},n_{3}+1\}=n_{2}.$
Thus, we have $\bar{n}_{1}(H^{(3)})+n_{2}(H^{(3)}) \leq l-2$. If $n_2+n_3 \leq l-2$, let $H^{(4)}$be the $(n_1,n_2)$-transformation of $G[A]$, then $n_1(H^{(4)})=n_1+1$, $n_{2}(H^{(4)})=\max\{n_{2}-1,n_{3}\}=n_{2}-1$ and $n_{3}(H^{(4)})=n_{3}.$
So, we get $n_2(H^{(4)})+n_3(H^{(4)})=n_2+n_3-1 \leq l-3<l-2.$ It follows from Claim 2 that neither $K_1 \vee H^{(3)}$ nor $K_1 \vee H^{(4)}$ contains $tP_l.$ As $n \geq 6.5025 \times 2^{(t-2)(l-1)+l-1+l-1+2} \geq 6.5025 \times 2^{s_2+2} $, by Lemma \ref{lm3}, we get $\rho(K_{1} \vee H^{(3)})>\rho(G)$ and $\rho(K_{1} \vee H^{(4)})>\rho(G)$, a contraction.
\end{proof}

Now we assert that $G[A] \cong H_{\mathcal{OP}}((t-1)l+l-2-x, x)$ for a fixed integer $x \in [1, \lfloor\frac{l-2}{2}\rfloor]$, and then $G \cong K_1 \vee H_{\mathcal{OP}}((t-1)l+l-2-x, x)$.  
Otherwise, we can obtain $H_{\mathcal{OP}}((t-1)l+l-2-x, x)$ by applying a series of $(s_1, s_2)$-transformations to $G[A]$, where $s_2 \leq s_1 \leq (t-1)(l-1) + l-2$.  
As $n \geq 6.5025 \times 2^{(t-2)(l-1)+l-1+l-1+2} \geq 6.5025 \times 2^{s_2+2} $, we have that $\rho < \rho(K_1 \vee H_{\mathcal{OP}}((t-1)l+l-2-x, x)$ by Lemma \ref{lm3}, which is impossible by the maximality of $\rho = \rho(G)$.

Next, we assert that $\rho(K_1 \vee H_{\mathcal{OP}}((t-1)l+l-2-x, x)) < \rho(K_1 \vee H^*) = \rho(K_1 \vee H_{\mathcal{OP}}(tl-l-1, l-1))$.
Let $P^{(i)}$ to denote the $i$-th longest path of $H_{\mathcal{OP}}((t-1)l+l-2-x, x)$ for any $i \in \{1, \ldots, q\}$ and $|P^{(i)}|$ to denote the orders of these $P^{(i)}$s, where $|P^{(1)}|=(t-1)l+l-2-x$, $|P^{(i)}|=x $ for $i \in \{2, \ldots, q-1\}$ and $|P^{(q)}| \geq 0$. For convenience, we write that $n_1^\circ=|P^{(1)}|$, $n_2^\circ=|P^{(i)}|(2 \leq i \leq q-1)$ and $n_3^\circ=|P^{(q)}|$ in the following text.
Assuming that $P^{(1)}=v_1v_2 \cdots v_{(t-1)l+l-2-x}$. Since $t \geq 2$, $l \geq 4$ and $n_1^\circ = (t-1)l+l-2-x$, there exists an edge $w'w''=v_{(t-1)l-1}v_{(t-1)l} \in E(P^{(1)})$. 
Note that $n-1 = \sum_{i=1}^{q} n_i \leq n_1 + (q-1)n_2$, it implies that $q-2 \geq n_3^\circ + 1$ by $n \geq 6.5025 \times 2^{(t-1)l+l-2+2} \gg (t-1)l+(l-2)+{\lfloor\frac{l-2}{2}\rfloor}^2+ 1$.  
Then $P^{(2)}, P^{(3)}, \ldots, P^{(n_3^\circ+2)}$ are paths of order $n_2^\circ$.  
And assume that $P^{(q)} = w_1w_2 \cdots w_{n_3^\circ}$.
Let $G'$ be obtained from $G$ by

(1) deleting $w'w''$ and joining $w''$ to an endpoint of $P^{(n_3^\circ+2)}$;

(2) deleting all edges of $P^{(q)}$ and joining $w_i$ to an endpoint of $P^{(i+1)}$ for each $i \in \{1, \ldots, n_3^\circ\}$, 
it means that $G'$ is obtained from $G$ by deleting $n_3^\circ$ edges and adding $n_3^\circ + 1$ edges.  
By Lemma \ref{lm4},  
\[  
\frac{1}{\rho^2} \leq x_{u_i}x_{u_j} \leq \frac{1}{\rho^2} + \frac{4.08}{\rho^3} + \frac{4.1616}{\rho^4} < \frac{1}{\rho^2} + \frac{5.08}{\rho^3}  
\]  
for any vertices $u_i, u_j \in A$. Therefore,  
\[  
\rho(G') - \rho \geq \frac{X^T(A(G') - A(G))X}{X^TX} > \frac{2}{X^TX} \left( \frac{(n_3^\circ + 1)}{\rho^2} - \frac{n_3^\circ}{\rho^2} - \frac{5.08n_3^\circ}{\rho^3} \right) > 0,  
\]  
where $n_3^\circ <\left\lfloor\frac{l-2}2\right\rfloor\leq\frac1{5.08}(\sqrt{n-1})\leq\frac1{5.08}\rho$ as $n \geq (5.08\left\lfloor\frac{l-2}2\right\rfloor)^2+1$ and $\rho \geq \rho(K_{1,n-1})=\sqrt{n-1}$ (since $K_{1,n-1}$ is a $tP_l$-free outerplanar graph for $l \geq 4$). So, $\rho(G^{\prime})>\rho.$
However, $P^{(1)}$ in $G'$ has only $(t-2)$ independent $l$-paths, which means that $r \leq t-2$. Then, by Subcase \ref{subcase2.1},
we have $\rho(K_1 \vee H_{\mathcal{OP}}((t-1)l+l-2-x, x) = \rho < \rho(G') \leq \rho(K_1 \vee H_{\mathcal{OP}}(tl-l-1, l-1))$,  
which is impossible by the maximality of $\rho = \rho(G)$, leading to a contradiction.  
\end{subcase}
\end{case}
\begin{case}{$G$ is a disconnected graph.}\label{case3.2}

Let $n \geq N + \frac{3}{2} + 3\sqrt{N - \frac{7}{4}}$, $N = \max\{1.27 \times 10^{7}, 6.5025 \times 2^{(t-2)(l-1)+2l}, (5.08\left\lfloor\frac{l-2}2\right\rfloor)^2+1\}$.
Suppose that $G$ is disconnected with components $G_1, G_2, \ldots, G_s$ ($s \geq 2$), and without loss of generality, assume $\rho(G_1) = \rho(G) = \rho$.
For each $i \in \{1, \ldots, s\}$, let $V(G_i) = \{u_{i,1}, u_{i,2}, \ldots, u_{i,n_i}\}$ and $t_i$ be the maximal number of independent $l$-paths of $G_i$.
Clearly, $\sum_{i=1}^s t_i \leq t-1$ and $\sum_{i=1}^s n_i = n$.
If $t_1 \leq t-2$, then constructing $G'$ by removing all the last edge of $P_l$s in $G_2, \ldots, G_s$, and then connecting just one of the endpoints of each $P_{l-1}$ in these components to $u_{1,1}$ in $G_1$ results in a connected $tP_l$-free outerplanar graph.
Since $G_1$ is a proper subgraph of $G'$, $\rho(G') > \rho(G_1) = \rho(G)$, contradicting the maximality of $\rho(G)$. Thus, $t_1 = t-1$, it means that $G_2,G_3,\ldots,G_s$ is $P_l$-free.
Furthermore, $G_1$ must be an extremal graph in SPEX$_\mathcal{OP}(n_1, tP_l)$, otherwise $G$ would not be an extremal graph in SPEX$_\mathcal{OP}(n, tP_l)$.
If $n_1 \geq N$, then $G_1 \cong K_1 \vee H(tl-l-1, l-1))$ for order $n_1$ by Case \ref{case3.1}.
And constructing a $G''$ by connecting all the longest paths in $G_2,G_3,\ldots,G_l$ to $K_1$ in $G_1$ results in a connected $tP_l$-free outerplanar graph, again contradicting the maximality of $\rho(G)$.
By Lemma \ref{lm7}, we obtain
$$\rho=\rho(G_1)\leq \frac{3}{2}+\sqrt{n_1-\frac{7}{4}} < \frac{3}{2}+\sqrt{N-\frac{7}{4}} \leq \sqrt{n-1}=\rho(K_{1, n-1})$$
as $n\geq N+\frac{3}{2}+3\sqrt{N-\frac{7}{4}}$ and $K_{1,n-1}$ is also $tP_l$-free for for $t \geq 1$ and $l \geq 4$ or $t \geq 2$ and $l \geq 3$.
This contradicts the maximality of $\rho=\rho(G)$.
Therefore, $G$ must be connected. By Case \ref{case3.1}, we conclude that $G \cong K_1 \vee H_{\mathcal{OP}}(tl-l-1,l-1)$, which completes the proof.
\end{case}
\vspace*{-6.56mm}
\end{proof}

\section{Proofs of Theorems \ref{thm3} and \ref{thm4}}\label{sc4}
In this section, we provide the detailed proofs of Theorems \ref{thm3} and \ref{thm4}.
When it comes to the study of planar graphs, the following structural lemmas are equally crucial in uncovering the properties of extremal graphs.

\begin{lemma}[Key Structure of Planar Extremal Graphs]\cite{X.L. Wang}\label{lm4}
Let $F$ be a planar graph not contained in $K_{2,n-2}$ where $n \geq \max\{2.67 \times 9^{17}, \frac{10}{9}|V(F)|\}$.  
Suppose that $G$ is a connected extremal graph in $\mathrm{SPEX}_{\mathcal{P}}(n, F)$ and
$X$
is the positive eigenvector of $\rho := \rho(G)$ with $\max_{v \in V(G)} x_{v} = 1$.   
Then the following two statements hold.  
 \begin{description}
   \item[(i)] There exist two vertices $u',u''\in V( G)$ such that $R:=N_G(u')\cap N_G( u'')=V( G)\setminus\{u', u''\}$ and $x_{u'}= x_{u''}= 1.$ 
In particular, $G$ contains a copy of $K_{2, n-2}.$
   \item[(ii)] The subgraph $G[R]$ of $G$ induced by $R$ is a disjoint union of some paths and cycles.
Moreover, if $G[R]$ contains a cycle then it is exactly a cycle, i.e., $G[R] \cong C_{n-2}$, and
if $u'u''\in E(G)$ then $G[R]$ is a disjoint union of some paths.
 \end{description}
 \end{lemma}

\begin{lemma}[Eigenvector Bounds]\cite{X.L. Wang}.\label{lm5}
Suppose further that G contains $K_{2, n- 2}$ as a subgraph. Let $u_{1}, u_{2}$ be the two vertices of $G$ that have degree $n-2$ in $K _{2, n- 2}$.
For any vertex $u\in V( G) \setminus \{ u_1, u_2\} $, we have 
$$\frac 2\rho \leq x_u\leq \frac 2\rho + \frac {4. 496}{\rho ^2}.$$
\end{lemma}

\begin{definition}[Planar Path Transformation]\cite{X.L. Wang}\label{de2}
Let $s_1$ and $s_2$ be two integers with $s_1\geq s_2\geq1$, and let $H=P_{s_1}\cup P_{s_2}\cup H_0$, where $H_0$ is a disjoint union of paths. We say that $H^*$ is an $(s_1,s_2)$-transformation of $H$ if
$$H^*:=\begin{cases}P_{s_1+1}\cup P_{s_2-1}\cup H_0&\text{if }s_2\geq2,\\P_{s_1+s_2}\cup H_0&\text{if }s_2=1.\end{cases}$$
Clearly, $H^*$ is a disjoint union of paths, which implies that $K_2 \vee H^*$ is planar.
\end{definition}

This transformation strictly also increases the spectral radius of \( K_2 \vee H \) for large \( n \).

\begin{lemma}[Spectral Comparison for Planar Transformations]\cite{X.L. Wang}\label{lm6}
  For \( H \) and \( H^* \) as in Definition \ref{de2}, if \( n \geq \max\{2.67 \times 9^{17}, 10.2 \times 2^{s_2} + 2\} \), then \( \rho(K_2 \vee H^*) > \rho(K_2 \vee H) \).
\end{lemma}
\vspace*{1mm}
\begin{proof}[\textbf{Proof of Theorem \ref{thm3}}]
Let $G$ be the extremal graph with $\textit{spex}_{\mathcal{P}}(n,P_{t\cdot l})$.
Noting that $F\in\{P_{t\cdot l}|t = 1, l \geq 6\}$ or $F\in\{P_{t\cdot l}|t = 2, l \geq 4\}$ or $F\in\{P_{t\cdot l}|t \geq 3, l \geq 3\}$ is not subgraph of $K_{2,n-2}$,
then $G$ contains a copy of $K_{2,n-2}$ for large enough $n$.
\begin{claiim}\label{claim4.1}
The vertices $u'$ and $u''$ are adjacent. Furthermore, the subgraph $G[R]$ of $G$ induced by $R$ is a disjoint union of some paths.  
\end{claiim}  
  
\begin{proof}  
Suppose to the contrary that $u'u'' \notin E(G)$. Then we assert that $G[R]$ contains at most $t - 1$ independent $(l - 1)$-paths.
If not, $G$ would contain $t$ edge-disjoint $P_l$ that intersect in one common vertex $u'$, which is impossible as $G$ is $P_{t\cdot l}$-free. We denote the edges of these independent $(l - 1)$-paths by $E(IP)$.
Define $G' = G + u'u'' - \sum_{u_iu_j \in E(IP)} u_iu_j$. Clearly, $G'$ is planar and $P_{t\cdot l}$-free. By Lemma \ref{lm5}, we have  
\[  
\begin{aligned}  
\rho(G') - \rho &\geq \frac{2}{X^TX} \left( x_{u'}x_{u''} - \sum_{u_iu_j \in E(IP)} x_{u_i}x_{u_j} \right) \\  
&\geq \frac{2}{X^TX} \left( 1 - (t - 1)(l - 2) \cdot \left( \frac{2}{\rho} + \frac{4.496}{\rho^2} \right)^2 \right) \\  
&\geq \frac{2}{X^TX} \left( 1 - (t - 1)(l - 2) \cdot \left( \frac{2}{\sqrt{2n - 4}} + \frac{4.496}{2n - 4} \right)^2 \right) \\  
&> 0,  
\end{aligned}  
\]  
as $n \geq \max\{2.67 \times 9^{17}, 10.2 \times 2^{(t - 1)(l - 1) + l - 3} + 2\}$ and $\rho \geq \rho(K_{2,n-2})=\sqrt{2n-4}$ (since $K_{2,n-2}$ is a $P_{t\cdot l}$-free planar graph).
Hence, $\rho(G') > \rho$, contrary to the maximality of $\rho = \rho(G)$. Therefore, $u'u'' \in E(G)$.
From Lemma \ref{lm4}, we establish that the subgraph $G[R]$ of $G$ induced by $R$ is a disjoint union of some paths.  
\end{proof}

By Claim \ref{claim4.1} and Lemma \ref{lm4}, $G\cong K_{2} \vee G[R]$ and $G[R]$ is a disjoint union of paths.
Suppose that $G[R]=\cup_{i=1}^qP_{n_i}$, where $q\geq2$ and $n_1\geq n_2\geq\cdots\geq n_q$.
Let $H$ be a disjoint union of $q$ paths. We use $n_i(H)$ to denote the order of the $i$-th longest path of $H$ for any $i\in\{1,...,q\}.$ 

(i) For $t=1$ and $l \geq 6$, we have the following claims.

\begin{claiim}\label{claim4.2}  
If $H$ is a disjoint union of paths, then $K_2 \vee H$ is $P_{l}$-free if and only if $n_1(H) +n_2(H) +n_3(H) \leq l-3$.  
\end{claiim} 

\begin{proof}  
One can observe that the length of the longest path in $K_2 \vee H$ is $n_1(H) +n_2(H) + n_3(H) + 2$, and $K_2 \vee (P_{n_1(H)} \cup P_{n_2(H)} \cup P_{n_3(H)})$ contains a path $P_i$ for each $i \in \{3, \ldots, n_1(H) +n_2(H) + n_3(H) + 2\}$.  
Thus, $n_1(H) +n_2(H) + n_3(H) + 2 \leq l - 1$ if and only if $K_2 \vee H$ is $P_l$-free, 
which the claim proves.  
\end{proof}  
  
By Claim \ref{claim4.2} and direct computation, we have $n_2 \leq n_1 \leq l - 3$ in $G[R]$, and so $10.2 \times 2^{l-3} +2 \geq 10.2 \times 2^{n_2}+ 2$. Thus,  
\begin{equation}\label{gongshin}  
n \geq \max\{2.67 \times 9^{17},  \frac{10}{9}|V(F)|, 10.2 \times 2^{l-3}+2\}.  
\end{equation}  
  
\begin{claiim}\label{claim4.3}  
$n_1 +n_2+n_3 = l - 3$.  
\end{claiim}  
  
\begin{proof}  
By Claim \ref{claim4.2}, we already know that $n_1+n_2+n_3 \leq l-2 $ in $G[R]$. Now suppose to the contrary that $n_1+n_2+n_3\leq l-3.$ Let $H^{\prime}$ be an $(n_1,n_t)$-transformation of $G[R].$ Clearly, $n_1(H^{\prime})=n_1+1$, $n_2(H^{\prime})=n_2$ and $n_3(H^{\prime})=n_3.$ Then, $n_1(H^{\prime})+n_2(H^{\prime})+n_2(H^{\prime})\leq l-3.$ By Claim 1, $K_2 \vee H^{\prime}$ is $P_l$-free. However, by (\ref{gongshin}) and Lemma \ref{lm6}, we have $\rho(K_2 \vee H^{\prime})>\rho$, contradicting by the maximality of $\rho = \rho(G)$. Hence, the claim holds.
\end{proof}

For integers $c$ and $d$, let $l-3 = 3c + d$ where $c \geq 1$ and $d= 0,1,2$.

\begin{claiim}\label{claim4.4}  
$n_i = c$ for $i=\{2,\ldots,q-1\}.$  
\end{claiim}  
  
\begin{proof}  
First, we assert that $n_i=n_3$ for $i=\{4,\ldots,q-1\}.$ Suppose to the contrary, then set $i_0=\min\{i|$3$\leq i\leq t-1,n_i\leq n_3-1\}.$ Let $H^\prime$ be an $(n_{i_0},n_q)$-transformation of $G[R].$ Clearly, $n_1(H^{\prime})=n_1$, $n_2(H^{\prime})=n_2$ and $n_3(H^{\prime})=\max\{n_3,n_{i_0}+1\}=n_3$, and so $n_1(H^{\prime})+n_2(H^{\prime})+n_3(H^{\prime})=l-3.$ By Claim \ref{claim4.2}, $K_2 \vee H^{\prime}$ is $P_l$-free. However, by (\ref{gongshin}) and Lemma \ref{lm6}, $\rho(K_2 \vee H^{\prime})>\rho$, contradicting by the maximality of $\rho = \rho(G)$. 
Next, we assert that $n_2=n_3.$ we already know that $n_2 \geq n_3$. Now suppose to the contrary that $n_2 > n_3$, it means that $n_2 \geq n_3+1 $. Let $H^{\prime}$ be an $(n_1,n_2)$-transformation of $G[R].$ Clearly, $n_1(H^{\prime})=n_1+1$, $n_2(H^{\prime})=\max\{n_2-1,n_3\}$ and $n_3(H^{\prime})=n_3.$ Then, $n_1(H^{\prime})+n_2(H^{\prime})+n_2(H^{\prime}) \leq  l-3.$ By Claim \ref{claim4.2}, $K_2 \vee H^{\prime}$ is $P_l$-free. However, by (\ref{gongshin}) and Lemma \ref{lm6}, we have $\rho(K_2 \vee H^{\prime})>\rho$, contradicting by the maximality of $\rho = \rho(G)$.

It follows that $G \cong K_2 \vee (P_{n_1}\cup(q-2)P_{n_2}\cup P_{n_q}).$
Finally, we claim that $n_2=c.$ Note that $1\leq n_2 \leq \left\lfloor\frac{l-3}3\right\rfloor.$ Then, $n_2=\left\lfloor\frac{l-3}3\right\rfloor$ for $l\in\{6,7,8\}.$ It remains the case $l\geq9.$ Suppose to the contrary, then $n_2\leq\lfloor\frac{l-6}3\rfloor$ as $n_2<\lfloor\frac{l-3}{3}\rfloor.$
Using $P^i$ to denote the $i$-th longest path of $G[R]$ for $i\in\{1,...,q\}.$
Noting that $n-2 = \sum_{i=1}^{q} n_i \leq n_1 + (q-1)n_2$, which implies that $q \geq n_4 + 4$ by $n \geq 10.2\times2^{(t-1)l+l-3}+2 \gg (t-1)l + l-1 + \lceil\frac{l-3}{2}\rceil^2 + 4\lfloor\frac{l-3}{2}\rfloor$.
Then $P^2,P^3,...,P^{n_2+3}$ are paths of order $n_2.$
We may assume that $P^{n_2+3}=w_1w_2\cdots w_{n_2}.$ Since $n_1 \geq \lfloor\frac{l-3}3\rfloor \geq 2$, there exists an endpoint $w^{\prime\prime}$ of $P^1$ with $w^\prime w^{\prime\prime}\in E\left(P^1\right).$
Let $G^\prime$ be obtained from $G$ by

(1) deleting $w^\prime w^{\prime\prime}$ and joining $w^{\prime\prime}$ to an endpoint of $P^{n_2+2};$

(2) deleting all edges of $P^{n_2+3}$ and joining $w_i$ to an endpoint of $P^{i+1}$ for each $i\in\{1,...,n_2\}.$
Then, $G^\prime$ is obtained from $G$ by deleting $n_2$ edges and adding $n_2+1$ edges. By Lemma \ref{lm5}, we obtain
$$\frac4{\rho^2}\leq x_{u_i}x_{u_j}\leq\frac4{\rho^2}+\frac{24}{\rho^3}+\frac{36}{\rho^4}<\frac4{\rho^2}+\frac{25}{\rho^3}$$
for any vertices $u_i,u_j\in A.$ Then
$$\rho(G^{\prime})-\rho\geq\frac{X^T(A(G^{\prime})-A(G))X}{X^TX}>\frac2{X^TX}\biggl(\frac{4(n_2+1)}{\rho^2}-\frac{4n_2}{\rho^2}-\frac{25n_2}{\rho^3}\biggr)>0,$$
where $n_2<\lfloor\frac{l-3}3\rfloor\leq\frac4{25}\sqrt{2n-4}\leq\frac4{25}\rho$ as $n\geq\frac{625}{32}\lfloor\frac{l-3}3\rfloor^2+2$ and $\rho \geq \rho(K_{2,n-2})=\sqrt{2n-4}$ (since $K_{2,n-2}$ is a $P_l$-free planar graph for $l \geq 6$). So, $\rho(G^{\prime})>\rho.$
Meanwhile, $G^\prime\cong K_2 \vee (P_{n_1-1}\cup(n_2+1)P_{n_2+1}\cup(q-n_2-4)P_{n_2}\cup P_{n_t}).$ By Claim \ref{claim4.2}, $G^\prime$ is $P_{l}$-free, contradicting the fact that $G$ is extremal to $spex_{P}(n,P_l).$ So, the claim holds.
\end{proof}

Since $n_1 + n_2 +n_3 = l-3 = 3c+d$ and $n_2=c=\lfloor \frac{l-3}{3} \rfloor$, we have $n_1=c+d=l-3-2\lfloor\frac{l-3}{3}\rfloor$. Moreover, since $n_i=\lfloor \frac{l-3}{3} \rfloor$ for $i \in \{2, \ldots, t - 1\}$ and $n_t \leq n_2$, we can see that $G[R] \cong H_{\mathcal{P}}(l-3-2\lfloor\frac{l-3}{3}\rfloor,\lfloor\frac{l-3}{3}\rfloor)$. This completes the proof of Theorem \ref{thm3} (i).

(ii) For $t=2$ and $l \geq 3$, then $P_{2\cdot l} \cong P_{2l-1}$. Clearly, by Theorem \ref{thm3} (ii), we can directly conclude that $G \cong K_2 \vee H_{\mathcal{P}}(l-2,l-2)$.

(iii) For $t \geq 5$ and $l \geq 3$, it is not difficult to know that a \(P_{t\cdot l}\)-free graph is also \(B_{tl}\)-free. Meanwhile, as we know from \cite{X.L. Yin}, the spectral extremal \(B_{tl}\)-free planar graph that is $K_2 \vee H_{\mathcal{P}}(tl-t-l,l-2)$, and $K_2 \vee H_{\mathcal{P}}(tl-t-l,l-2)$ happens to be \(P_{t\cdot l}\)-free as well. Therefore, it is proven that the extremal graph of \(P_{t\cdot l}\)-free is $K_2 \vee H_{\mathcal{P}}(tl-t-l,l-2)$ for $t \geq 5$ and $l\geq 3$. 
\end{proof}

\vspace*{1mm}
Our complete proof of Theorem \ref{thm4} also needs an upper bound of the spectral radius of planar graphs as following.

\begin{lemma}\cite{M.N. Ellingham}.\label{lm8}
Let $G$ be a planar graph on $n\geq3$ vertices. Then $\rho\left(G\right)\leq 2+\sqrt{2n-6}.$
\end{lemma}
\vspace*{1mm}

\begin{proof}[\textbf{Proof of Theorem \ref{thm4}}]

Let $G$ be the extremal graph with $\textit{spex}_{\mathcal{P}}(n,tP_{l})$.
By Lemma \ref{lm4}, noting that $F\in\{tP_{l}|t = 1, l \geq 6\}$ or $F\in\{tP_{l}|t = 2, l \geq 4\}$ or $F\in\{tP_{l}|t \geq 3, l \geq 3\}$ is not subgraph of $K_{2,n-2}$,
then $G$ contains a copy of $K_{2,n-2}$.

(i) For $t=1$ and $l \geq 6$, the extremal graph be the same with SPEX$_{\mathcal{P}}(n,P_{1\cdot l})$. Therefore, we get Theorem \ref{thm4} (i).

(ii) For $t = 2$, $l \geq 4$ and $t \geq 3$, $l \geq 3$, we have the following claims.
\begin{claiiim}\label{claim5.1}  
The vertices $u'$ and $u''$ are adjacent. Furthermore, the subgraph $G[R]$ of $G$ induced by $R$ is a disjoint union of some paths.  
\end{claiiim}  
  
\begin{proof}  
Suppose to the contrary that $u'u'' \notin E(G)$. Then we assert that $G[R]$ contains at most $t-1$ independent $l$-paths.
If not, $G$ would contain $t$ vertex-disjoint $P_l$, which is impossible as $G$ is $tP_{l}$-free. We denote the edges of these independent $l$-paths by $E(IP)$.
Define $G' = G + u'u'' - \sum_{u_iu_j \in E(IP)} u_iu_j$. Clearly, $G'$ is planar and $tP_{l}$-free. By Lemma \ref{lm4}, we have  
\[  
\begin{aligned}  
\rho(G') - \rho &\geq \frac{2}{X^TX} \left( x_{u'}x_{u''} - \sum_{u_iu_j \in E(IP)} x_{u_i}x_{u_j} \right) \\  
&\geq \frac{2}{X^TX} \left( 1 - (t - 1)(l - 1) \cdot \left( \frac{2}{\rho} + \frac{4.496}{\rho^2} \right)^2 \right) \\  
&\geq \frac{2}{X^TX} \left( 1 - (t - 1)(l - 1) \cdot \left( \frac{2}{\sqrt{2n - 4}} + \frac{4.496}{2n - 4} \right)^2 \right) \\  
&> 0,  
\end{aligned}  
\]  
as $n \geq \max\{2.67 \times 9^{17}, 10.2 \times 2^{(t - 1)l + l - 3} + 2\}$ and $\rho \geq \rho(K_{2,n-2})=\sqrt{2n-4}$ (since $K_{2,n-2}$ is a $tP_l$-free planar graph for $l \geq 6$).
Hence, $\rho(G') > \rho$, contrary to the maximality of $\rho = \rho(G)$. Therefore, $u'u'' \in E(G)$.
From Lemma \ref{lm4}, we establish that the subgraph $G[R]$ of $G$ induced by $R$ is a disjoint union of some paths.  
\end{proof}

From the Claim \ref{claim5.1}, we can assume that $G[R]=\cup_{i=1}^qP_{n_i}$, where $q\geq2$ and $n_1\geq n_2\geq\cdots\geq n_q.$
Let $H$ be an union of disjoint paths.
Using $n_i(H)$ to denote the order of the $i$-th longest path of $H$ for any $i\in\{1,...,q\}.$ 

Now we are ready to complete the proof of Theorem \ref{thm4} (ii).
Let $H^* = H_{\mathcal{P}}(l-1, \lceil\frac{l-2}{2}\rceil, \lfloor\frac{l-2}{2}\rfloor)$. Clearly, $K_2 \vee H^* = K_2 \vee H_{\mathcal{P}}(l-1, \lceil\frac{l-2}{2}\rceil, \lfloor\frac{l-2}{2}\rfloor)$, which means that $n_1 = l-1$, $n_2 = \lceil\frac{l-2}{2}\rceil$ and $n_3 = \lfloor\frac{l-2}{2}\rfloor$.  
If $G[R] \cong H^*$, then $G \cong K_2 \vee H_{\mathcal{P}}(l-1, \lceil\frac{l-2}{2}\rceil,\lfloor\frac{l-2}{2}\rfloor)$, and we are done.  
If $G[R] \ncong H^*$, by Claim \ref{claim5.1} and Lemma \ref{lm4}, then $G[R]$ is a disjoint union of some paths.
Now, consider the quantity $r$ which represents the number of independent $l$-paths in $P_{n_1}$. 
Firstly, we can get $r=0$ for $l=4,5$ and $r \leq 1$ for $l \geq 6$. Otherwise, there would be a $2P_l$ in $G$, a contradiction.
Next, we will follow the following cases.
  
\begin{casee}\label{case4.1}  
$r = 0$ and $l \geq 4$. Then we have the following claims.

\begin{claiiim}\label{claim5.2}  
Let $H$ be an union of $q$ vertex–disjoint paths.
Then $K_2 \vee H$ does not contain $2P_l$ if and only if 
$n_1(H)+n_2(H)+n_3(H) \leq 2l-3$ and $n_3(H) +n_4(H) \leq l-2.$
\end{claiiim}

\begin{proof}
We undertake the verification of the converse of Claim \ref{claim5.2}, that is, $K_2 \vee H$ contains $2P_l$ if and only if $n_3(H) +n_4(H) \geq l-1$ or $n_1(H) +n_2(H) +n_3(H) \geq 2l-2.$
Let $V(K_2)=\{u',u''\}$, $V(H)=\{v_i: 1 \leq i \leq|V(H)|\}$ and note that $n_1(H) \leq l-1.$

Let's first prove the sufficiency of the converse of Claim \ref{claim5.2}.
We assume that $K_2 \vee H$ contains a subgraph $2P_l$. Let $P^{(1)}$ and $P^{(2)}$ be the two $l$-paths of $2P_l.$ And note that $l-1 \geq n_1(H) \geq n_2(H) \geq \cdots \geq n_q(H)$, then $P_l$ is not subgraph of any $P_{n_i(H)}$ for $i \in \{1,2,\ldots,q\}.$
Two cases are considered:
(i) There exits $P_l$ is a subgraph of $(P_{n_i(H)} \cup P_{n_j(H)}) \vee \{u'\}$ (or $\{u''\}$) for $3 \leq i < j \leq q$, then $n_3(H) +n_4(H) \geq n_i(H) +n_j(H) \geq l-1.$ Furthermore, $n_1(H) + n_2(H) \geq n_3(H) +n_4(H) \geq l-1$, then we can easily find the other $P_l$ in $(P_{n_1(H)} \cup P_{n_2(H)}) \vee \{u''\}$ (or $\{u'\}$).
(ii) There is no $P_l$ in $(P_{n_3(H)} \cup P_{n_4(H)}) \vee \{u'\}$(or $\{u''\}$), then $n_3(H) +n_4(H) \leq l-2$. Recall that $l-1 \geq n_1(H) \geq n_2(H) \geq \cdots \geq n_q(H) \geq 1$, then there is no $2P_l$ in $(P_{n_1(H)} \cup P_{n_2(H)}) \vee \{u'\}$(or $\{u''\}$). Therefore, $P^{(1)}$ and $P^{(2)}$ must be in $(P_{n_1(H)} \cup P_{n_2(H)} \cup P_{n_3(H)}) \vee \{u',u''\}.$ Then $n_1(H) +n_2(H)+n_3(H) \geq 2l-2.$
And now we prove the necessity of the converse of Claim \ref{claim4.2}.
If $n_1(H) +n_2(H) \geq n_3(H) +n_4(H) \geq l-1$, then $K_2 \vee H$ contains two independent $P_l$ as subgraph, where one $P_l$ in $(P_{n_1(H)} \cup P_{n_2(H)}) \vee \{u'\}$ and the other in $(P_{n_3(H)} \cup P_{n_4(H)}) \vee \{u''\}.$
If $n_3(H) +n_4(H) \leq l-2$ but 
$n_1(H) + n_2(H) + n_3(H) \geq 2l-2,$ then $K_2 \vee H$ contains two independent $P_l$s as subgraph, where two $P_l$s both in $(P_{n_1(H)} \cup P_{n_2(H)} \cup P_{n_3}) \vee \{u',u''\}.$
Hence, the claim holds.
\end{proof}

\begin{claiiim}\label{claim5.3}  
$n_1=l-1.$
\end{claiiim}

\begin{proof}
By Claim \ref{claim5.2}, we already know that $n_1 \leq l-1$ as $r=0$. Suppose to the contrary that $n_1 < l-1$, it means that $n_1 \leq l-2.$
Let $H^{\prime}$ be an $(n_1,n_2)$-transformation of $G[R].$ Clearly, $n_1(H^{\prime})=n_1+1$, $n_2(H^{\prime})= max\{n_2-1, n_3\}$, $n_3(H^{\prime}) = min\{n_2-1, n_3\}$. Note that $n_4 \leq n_3 \leq n_2 \leq n_1 \leq l-2$, then we get $n_1(H^{\prime})+n_2(H^{\prime}) +n_3(H^{\prime}) \leq 2l-3$ and $n_3(H^{\prime}) +n_4(H^{\prime}) \leq l-2.$ By Claim \ref{claim5.2}, $K_2 \vee H^{\prime}$ is $2P_l$-free.
As $n \geq 10.2 \times 2^{l-1} + 2 \geq 10.2 \times 2^{s_2} + 2$, we conclude that $\rho < \rho(K_2 \vee H^{\prime})$ by Lemma \ref{lm6},  
which is impossible by the maximality of $\rho = \rho(G)$. 
\end{proof}

\begin{claiiim}\label{claim5.4}  
$n_2 +n_3 = l - 2$.  
\end{claiiim}  
  
\begin{proof}  
By Claims \ref{claim5.2} and \ref{claim5.3}, we already know that $n_2+n_3 \leq l-2 $.
Now suppose to the contrary that $n_2+n_3 < l-2.$ Let $H^{\prime}$ be an $(n_2,n_q)$-transformation of $G[R].$ Clearly, $n_2(H^{\prime})=n_2+1$ and $n_3(H^{\prime})=n_3.$ Then, $n_2(H^{\prime})+n_3(H^{\prime})\leq l-2.$
Note that $n_4 \leq n_3 \leq n_2 \leq n_1 = l-1$, then we get $n_1(H^{\prime})+n_2(H^{\prime}) +n_3(H^{\prime}) \leq 2l-3$ and $n_3(H^{\prime}) +n_4(H^{\prime}) \leq l-2.$ By Claim \ref{claim5.2}, $K_2 \vee H^{\prime}$ is $2P_l$-free. As $n \geq 10.2 \times 2^{2l-3} + 2 \geq 10.2 \times 2^{s_2} + 2$ and Lemma \ref{lm6}, $\rho(K_2 \vee H^{\prime})>\rho$, contradicting that $G$ is extremal to $spex_{\mathcal{P}}(n,2P_{l}).$ Hence, the claim holds.
\end{proof}  

\begin{claiiim}\label{claim4.5}  
$n_i = \lfloor \frac{l-2}{2} \rfloor$ for $i \in \{3, \ldots, q - 1\}$.  
\end{claiiim}  
  
\begin{proof}  
Suppose to the contrary, then set $i_0= \min\{i|4\leq i\leq t-1,n_i\leq n_3-1\}.$ Let $H^\prime$ be an $(n_{i_0},n_t)$-transformation of $G[R].$ Clearly, $n_2(H^{\prime})=n_2$ and $n_3(H^{\prime})=\max\{n_3,n_{i_0}+1\}=n_3$, and so $n_2(H^{\prime})+n_3(H^{\prime})=l-2.$
By Claim \ref{claim5.2}, $K_2 \vee H^{\prime}$ is $2P_l$-free. As $n \geq 10.2 \times 2^{(t-1)(l-1)+l-3} + 2 \geq 10.2 \times 2^{s_2} + 2$, by Lemma \ref{lm6}, we have $\rho(K_2 \vee H^{\prime})>\rho$, contradicting the fact that $G$ is extremal to spex$_{P}(n,2P_l).$ Hence, the claim holds. It follows that $G=K_2 \vee (P_{n_1}\cup P_{n_2}\cup(q-3)P_{n_3}\cup P_{n_t}).$
Finally, we claim that $n_3=\left\lfloor\frac{l-2}2\right\rfloor.$ Note that $1\leq n_3\leq\left\lfloor\frac{l-2}2\right\rfloor.$ Then, $n_3=\left\lfloor\frac{l-2}2\right\rfloor$ for $l\in\{4,5\}.$ It remains the case $l\geq6.$
Suppose to the contrary, then $n_3\leq\left\lfloor\frac{l-4}2\right\rfloor$ as $n_3<\left\lfloor\frac{l-2}{2}\right\rfloor.$ Using $P^i$ to denote the $i$-th longest path of $G[R]$ for $i\in\{1,...,q\}.$
Recall that $q\geq n_3+4.$
Noting that $n-2 = \sum_{i=1}^{q} n_i \leq n_1 + (q-1)n_2$, which implies that $q \geq n_4 + 4$ by $n \geq 10.2\times2^{(t-1)l+l-3}+2 > (t-1)l + l-1 + \lceil\frac{l-3}{2}\rceil^2 + 4\lfloor\frac{l-3}{2}\rfloor$.
Then $P^3,P^4,...,P^{n_3+3}$ are paths of order $n_3.$ We may assume that $P^{n_3+3}=w_1w_2\cdots w_{n_3}.$
Since $n_1=l-3-n_2\geq\left\lceil\frac{l-2}2\right\rceil\geq2$, there exists an endpoint $w^\prime$ of $P^1$ with $w^\prime w^{\prime\prime}\in E\left(P^1\right).$
Let $G^\prime$ be obtained from $G$ by

(1) deleting $w^\prime w^{\prime\prime}$ and joining $w^\prime$ to an endpoint of $P^{n_3+2};$

(2) deleting all edges of $P^{n_3+3}$ and joining $w_i$ to an endpoint of $P^{i+2}$ for each $i\in\{1,...,n_3\}.$
Then, $G^\prime$ is obtained from $G$ by deleting $n_3$ edges and adding $n_3+1$ edges.
By Lemma \ref{lm5}, we obtain
$$\frac4{\rho^2}\leq x_{u_i}x_{u_j}\leq\frac4{\rho^2}+\frac{24}{\rho^3}+\frac{36}{\rho^4}<\frac4{\rho^2}+\frac{25}{\rho^3}$$
for any vertices $u_i,u_j\in R.$ Then

$$\rho(G^{\prime})-\rho\geq\frac{X^T(A(G^{\prime})-A(G))X}{X^TX}>\frac2{X^TX}\biggl(\frac{4(n_3+1)}{\rho^2}-\frac{4n_3}{\rho^2}-\frac{25n_3}{\rho^3}\biggr)>0,$$
where $n_3<\left\lfloor\frac{l-2}2\right\rfloor\leq\frac4{25}\sqrt{2n-4}\leq\frac4{25}\rho$ as $n\geq\frac{625}{32}\left\lfloor\frac{l-2}2\right\rfloor^2+2$ and $\rho \geq \rho(K_{2,n-2})=\sqrt{2n-4}$ (since $K_{2,n-2}$ is a $2P_l$-free planar graph).
So, $\rho(G^{\prime})>\rho.$
Meanwhlie, $G^\prime\cong K_2 \vee (P_{n_1}\cup P_{n_2-1}\cup(n_3+1)P_{n_3+1}\cup(q-n_3-4)P_{n_3}\cup P_{n_q}).$
By Claim \ref{claim5.2}, $G^\prime$ is $P_{l}$-free, contradicting the fact that $G$ is extremal to $spex_{P}(n,P_l).$ So, the claim holds.
\end{proof}

Since $n_2 + n_3 = l-2$ and $n_3=\lfloor \frac{l-2}{2} \rfloor$, we have $n_2=\lceil \frac{l-2}{2}\rceil$. Moreover, since $n_i=\lfloor \frac{l-2}{2} \rfloor$ for $i \in \{3, \ldots, q - 1\}$ and $n_q \leq n_3$, we can see that $G[R] \cong H_{\mathcal{P}}(l-1, \lceil \frac{l-2}{2} \rceil, \lfloor\frac{l-2}{2}\rfloor)$ for case of $r=0$. Furthermore, we can conclude that $G\cong H_{\mathcal{P}}(l-1, \lceil \frac{l-2}{2} \rceil, \lfloor\frac{l-2}{2}\rfloor)$.

\end{casee}  
  
\begin{casee}\label{case4.2}
$r = 1$ and $l \geq 6$. Then we have the following claims.
\begin{claiiim}\label{claim5.6}  
Let $P_{n_1} = v_1v_2 \cdots v_{n_1}$, $P' = v_1v_2 \cdots v_{l}$ and $P_{\bar{n}_1} = v_{l+1} \cdots v_{n_1}$, where $\bar{n}_1=n_1-l$.
Then $K_2 \vee G[A]$ does not contain $2P_l$ if and only if $\bar{n}_1+n_2+n_3 \leq l-3$ or $n_2+n_3+n_4 \leq l-3$.
\end{claiiim}

\begin{proof}
Observing that the longest path which belongs to $K_2 \vee (H-P')$ has an order of $\bar{n}_1 + n_2 + n_3 + 2$ or $n_2+n_3+n_4+2$.
Furthermore, $K_2 \vee (P_{\bar{n}_1} \cup P_{n_2} \cup P_{n_3})$ or $K_2 \vee (P_{n_2} \cup P_{n_3} \cup P_{n_4})$ contains a path $P_i$ for each $i \in \{3,\ldots,\max\{\bar{n}_1+ n_2 +n_3 + 2, n_2+n_3 +n_4+2\}\}.$
Therefore, $\bar{n}_1+n_2+n_3 \leq l-3$ or $n_2+n_3+n_4 \leq l-3$ if and only if $K_2 \vee H$ is $2P_{l}$‐free.
Hence, the claim holds.
\end{proof}

\begin{claiiim}\label{claim5.7}  
$n_i=n_3$ for $i \in \{4,\ldots,q-1\}$.
\end{claiiim}

\begin{proof}
Suppose to the contrary, then there exists an $i_0 = \min\{i|4 \leq i \leq q-1\}$.
If $i_0 > 4$, then let $H'$ be an $(n_{i_0},n_q)$-transformation. Thus, $n_1(H')=n_1$(it means $\bar{n}_1(H') = \bar{n}_1$), $n_2(H')=n_2$, $n_3(H')=\max\{n_3,n_{i_0}+1\}$ and $n_4(H')=n_4$. Therefore, $\bar{n}_1(H')+n_2(H')+n_3 \leq l-3$ or $n_2(H')+n_3(H')+n_4(H') \leq l-3$. By Claim \ref{claim5.6}, $K_2 \vee H'$ is $2P_{l}$-free.
As $n \geq 10.2 \times 2^{(t-1)(l-1)+l-3} + 2 \geq 10.2 \times 2^{s_2} + 2$, by Lemma \ref{lm6}, we have $\rho(K_2 \vee H') > \rho$, a contradiction.
If $i_0 =4$, we first suppose that there exists a $j$ such that $n_j < n_4 \leq n_3-1$, where $5 \leq j \leq q-1$. Let $H''$ be the $(n_j,n_t)$-transformation of $G[R]$, where $5 \leq j \leq q-1.$ Then $\bar{n}_{1}(H'') = \bar{n}_{1}$, $n_{2}(H'') = n_{2}$, $n_3(H'') = n_3$ and $n_{4}(H'')=\max\{n_{4},n_{j}+1\}=n_{4}.$
Using the methods similar to those for $i_0>4$, we can also get a contradiction.
Then we suppose that $n_j=n_4$ for any $j\in\{5,\cdots,q-1\}$. If $\bar{n}_1+n_2+n_3 \leq l-3$, let $H^{(3)}$ be the $(n_4,n_q)$-transformation of $G[R]$, then $n_{1}(H^{(3)})=n_{1},n_{2}(H^{(3)})=n_{2},n_{3}(H^{(3)})=\max\{n_{3},n_{4}+1\}=n_{3}.$ 
Thus, we have $\bar{n}_{1}(H^{(3)})+n_{2}(H^{(3)})+n_{3}(H^{(3)}) \leq l-3$. If $n_2+n_3+n_4 \leq l-3$, let $H^{(4)}$be the $(n_1,n_2)$-transformation of $G[R]$, then $n_1(H^{(4)})=n_1+1$, $n_{2}(H^{(4)})=\max\{n_{2}-1,n_{3}\}=n_{2}-1$, $n_3(H^{(4)})=n_3$ and $n_{4}(H^{(4)})=n_{4}.$
So, we get $n_2(H^{(4)})+n_3(H^{(4)})+n_4(H^{(4)})=n_2+n_3+n_4-1 \leq l-4<l-3.$ It follows from Claim \ref{claim5.6} that neither $K_2 \vee H^{(3)}$ nor $K_2 \vee H^{(4)}$ contains $2P_{l}.$ As $n \geq 10.2 \times 2^{(t-1)(l-1)+l-3} + 2 \geq 10.2 \times 2^{s_2} + 2$, by Lemma \ref{lm5}, we get $\rho(K_{2} \vee H^{(3)})>\rho(G)$ and $\rho(K_{2} \vee H^{(4)})>\rho(G)$, a contraction.
\end{proof}

For $l\geq6$ and fixed two integers $1 \leq x \leq \lceil\frac{l-3}{3}\rceil$ and $1 \leq y \leq \lfloor\frac{l-3}{3}\rfloor$, we said $G[R] \cong H_{\mathcal{P}}(l+l-3-x-y,x,y)$.
If not,
we can obtain $G[R] \cong H_{\mathcal{P}}(l+l-3-x-y,x,y)$ by applying a series of $(s_1, s_2)$-transformations to $G[R]$, where $s_2 \leq s_1 \leq (t-1)l+l-1+l-1$.
As $n \geq 10.2 \times 2^{(t-1)(l-1)+l-3} + 2 \geq 10.2 \times 2^{s_2} + 2$, we conclude that $\rho < \rho(K_2 \vee H_{\mathcal{P}}(l+l-3-x-y,x,y)$ by Lemma \ref{lm5},  
which is impossible by the maximality of $\rho = \rho(G)$.
Now we assert that $\rho(K_2 \vee H_{\mathcal{P}}(l+l-3-x-y,x,y) < \rho(K_2 \vee H^*) = \rho(K_2 \vee H_{\mathcal{P}}(l-1,\lceil\frac{l-2}{2}\rceil,\lfloor\frac{l-2}{2}\rfloor)$.
Let $P^{(i)}$ to denote the $i$-th longest path of $H_{\mathcal{P}}(l+l-3-x-y,x,y)$ for any $i \in \{1, \ldots, q\}$ and $|P^{(i)}|$ to denote the orders of these $P^{(i)}$s, where $|P^{(1)}|=l+l-3-x-y$, $|P^{(2)}|=x$, $|P^{(i)}|=y $ for $i \in \{3, \ldots, q-1\}$ and $|P^{(q)}| \geq 0$. For convenience, we write that $n_1^\circ=|P^{(1)}|$, $n_2^\circ=|P^{(2)}|$, $n_3^\circ=|P^{(i)}|(3 \leq i \leq q-1)$ and $n_4^\circ=|P^{(q)}|$ in the following text.
As $l \geq 6$ and $n_1^\circ = l+l-3-x-y $, there are two edges $w'w''=v_{l-1}v_{l}$ and $z'z'' =$ $v_{l+\lceil\frac{l-2}{2}\rceil-x-1}$ $v_{l+\lceil\frac{l-2}{2}\rceil-x}$ $\in E(P^{(1)})$.
Noting that $n-2 = \sum_{i=1}^{q} n_i \leq n_1 + (q-1)n_2$, which implies that $q \geq n_4^\circ + 4$ by $n \geq 10.2\times2^{(t-1)l+l-3}+2 > l + l-1 + \lceil\frac{l-3}{2}\rceil^2 + 4\lfloor\frac{l-3}{2}\rfloor$.  
Then $P^{(3)}, P^{(4)}, \ldots, P^{(n_4^\circ+3)}$ are paths of order $n_3^\circ$.  
And assume that $P^{(q)} = w_1w_2 \cdots w_{n_4^\circ}$.
Let $G'$ be obtained from $G$ by 

(1) deleting $w'w''$ and joining $w''$ to an endpoint of $P^{(2)}$;

(2) deleting $z'z''$ and joining $z''$ to an endpoint of $P^{(3)}$;

(3) deleting all edges of $P^{(q)}$ and joining $w_i$ to an endpoint of $P^{(i+3)}$ for each $i \in \{1, \ldots, n_4^\circ\}$. 
Clearly, $G'$ is obtained from $G$ by deleting $n_4^\circ+1$ edges and adding $n_4^\circ + 2$ edges.  
By Lemma \ref{lm4}, then  
\[  
\frac{4}{\rho^2} \leq x_{u_i}x_{u_j} \leq \frac{4}{\rho^2} + \frac{24}{\rho^3} + \frac{36}{\rho^4} < \frac{4}{\rho^2} + \frac{25}{\rho^3}  
\]  
for any vertices $u_i, u_j \in R$. Therefore,  
\[  
\rho(G') - \rho \geq \frac{X^T(A(G') - A(G))X}{X^TX} > \frac{2}{X^TX} \left( \frac{4(n_4^\circ + 2)}{\rho^2} - \frac{4(n_4^\circ+1)}{\rho^2} - \frac{25(n_4^\circ+1)}{\rho^3} \right) > 0,  
\]  
where $n_4^\circ + 1 <  y +1 \leq \lfloor \frac{l-3}{3} \rfloor+1 \leq \frac{4}{25} \sqrt{2n-4} \leq \frac{4}{25} \rho$ as $n \geq \frac{625}{32} (\lfloor \frac{l-3}{3} \rfloor + 1 )^2 + 2$ and $\rho \geq \rho(K_{2,n-2})=\sqrt{2n-4}$ (since $K_{2,n-2}$ is a $2P_l$-free planar graph).
Yet, it is sad to say that $P_{n_1}$ in $G'$ has no $l$-paths, which means that $r = 0$. Then, by Case \ref{case4.1},
we have $\rho = \rho(K_2 \vee H_{\mathcal{P}}(l+l-3-x-y,x,y) < \rho(G') \leq \rho(K_2 \vee H_{\mathcal{P}}(l-1,\lceil\frac{l-2}{2}\rceil, \lfloor\frac{l-2}{2}\rfloor))$,  
which is impossible by the maximality of $\rho = \rho(G)$, leading to a contradiction.  
\end{casee}

(iii) It remains that $t\geq 3$ and $l \geq 3$, we consider the following claims in order to provide the remaining proof for Theorem \ref{thm4}.

 \begin{claiiim}\label{claim5.8}
If $H$ is a union of disjoint paths and $K_2 \vee H$ is $tP_l$-free,
then $n_1(H)+n_2(H)+n_3(H) \leq n_1(H)+2n_2(H) \leq (t-3)l+l-1+l-1+l-1$.
\end{claiiim}

\begin{proof}
Suppose to the contrary that $n_1(H) + n_2(H) +n_3(H) \geq (t - 3)l + 2(l - 1)+ l$.  
Let $a$ and $b$ be the maximum number of independent $l$-paths in $P_{n_1(H)}$, $P_{n_2(H)}$ and $P_{n_3(H)}$, respectively, i.e.,  
\[  
a = \left\lfloor \frac{n_1(H)}{l} \right\rfloor, \quad b = \left\lfloor \frac{n_2(H)}{l} \right\rfloor \quad \text{and} \quad c = \left\lfloor \frac{n_3(H)}{l} \right\rfloor.  
\]  
This implies that $n_1(H) - al \leq l - 1$, $n_2(H) - bl \leq l - 1$ and $n_3(H) - cl \leq l - 1$.
Next, we assert that $a + b + c \geq t - 2$.
If $a + b +c < t - 2$, recall that $n_1(H) + n_2(H) +n_2(H) \geq (t - 3)l + 2(l - 1)+ l$, then 
\[  
n_1(H) + n_2(H) +n_3(H)- al - bl -cl \geq 3l - 2.  
\]  
On the other hand, since $n_1(H) - al \leq l - 1$, $n_2(H) - bl \leq l - 1$ and $n_3(H) - cl \leq l - 1$, we have  
\[  
n_1(H) + n_2(H) +n_3(H)- al - bl -cl \leq 3l - 3 < 3l - 2,  
\]
which is a contradiction. Thus, $a + b + c \geq t - 2$. Therefore, $G[P_{n_1(H)} \cup P_{n_2(H)} \cup \{u', u''\}]$ would generate a $tP_l$ due to $n_1(H) + n_2(H) +n_3(H) \geq (t - 3)l + 2(l - 1)+ l$, a contradiction.   
\end{proof}

Let $H^* = H_{\mathcal{P}}(tl-2l-1, l-1)$. Clearly, $K_2 \vee H^* = K_2 \vee H_{\mathcal{P}}((t-3)l+l-1, l-1)$, which means that $n_1 = (t-3)l+l-1$ and $n_2 = l-1$.  
If $G[R] \cong H^*$, then $G \cong K_2 \vee H_{\mathcal{P}}(tl-2l-1, l-1)$, and we are done.  
If $G[R] \ncong H^*$, by Lemma \ref{lm4} and Claim \ref{claim5.8}, then $G[R]$ is a disjoint union of some paths and $n_1 + n_2 + n_3 \leq (t-3)l + l-1 + l-1 + l-1$.  
Now, consider the quantity $r$ which represents the number of independent $l$-paths in $P_{n_1}$.
Firstly, we can get $r\leq t-3$ for $l=3$, $r \leq t-2$ for $l=4,5$ and $r \leq t-1$ for $l \geq 6$. Otherwise, there would be a $tP_l$ in $G$, a contradiction.
Next, we will follow the following cases.
  
\begin{caseee}\label{case4.3}  
$r \leq t-3$ and $l \geq 3$. We can obtain $H^*$ by applying a series of $(s_1, s_2)$-transformations to $G[R]$, where $s_2 \leq s_1 \leq (t-1)(l-1) + l-3$.  
As $n \geq 10.2 \times 2^{(t-1)(l-1)+l-3} + 2 \geq 10.2 \times 2^{s_2} + 2$, we conclude that $\rho < \rho(K_2 \vee H^*) = \rho(K_2 \vee H_{\mathcal{P}}((t-3)l+l-1, l-1))$ by Lemma \ref{lm6},  
which is impossible by the maximality of $\rho = \rho(G)$.
Up to this case, we can conclude that $G\cong K_2 \vee H_{\mathcal{P}}((t-3)l+l-1,l-1)$.
\end{caseee}

\begin{caseee}  
$r = t - 2$ and $l \geq 4$. We denote that $P_{n_1} = v_1v_2 \cdots v_{n_1}$, $P' = v_1v_2 \cdots v_{(t-2)l}$, $P_{\bar{n}_1} = v_{(t-2)l+1} \cdots $ $v_{n_1}$. 
Then $K_2 \vee G[R]$ is $tP_l$ is equivalent to $K_2 \vee \bar{H}$ is $2P_l$, where $K_2 \vee \bar{H} \cong K_2 \vee (G[R]-P')$.
Therefore, similar to Case \ref{case4.1} in the proof process of Theorem \ref{thm4} (ii), we have
$\bar{n}_1=l-1$, $n_2=\lceil\frac{l-2}{2}\rceil$ and $n_i=n_3=\lfloor\frac{l-2}{2}\rfloor$ for $i \in \{4,\ldots,q-1\}$, it means that $G[R] \cong H_{\mathcal{P}}((t-2)l+l-1,\lceil\frac{l-2}{2}\rceil,\lfloor\frac{l-2}{2}\rfloor)$.

Next, we assert that $\rho(K_2 \vee H_{\mathcal{P}}((t-2)l+l-1,\lceil\frac{l-2}{2}\rceil,\lfloor\frac{l-2}{2}\rfloor) < \rho(K_2 \vee H^*) = \rho(K_2 \vee H_{\mathcal{P}}((t-3)l+l-1,l-1)$.
Let $P^{(i)}$ to denote the $i$-th longest path of $H_{\mathcal{P}}((t-2)l+l-1,\lceil\frac{l-2}{2}\rceil,\lfloor\frac{l-2}{2}\rfloor)$ for any $i \in \{1, \ldots, q\}$ and $|P^{(i)}|$ to denote the orders of these $P^{(i)}$s, where $|P^{(1)}|=(t-2)l+l-1$, $|P^{(2)}|=\lceil\frac{l-2}{2}\rceil$, $|P^{(i)}|=\lfloor\frac{l-2}{2}\rfloor $ for $i \in \{3, \ldots, q-1\}$ and $|P^{(q)}| \geq 0$. For convenience, we write that $n_1^\circ=|P^{(1)}|$, $n_2^\circ=|P^{(2)}|$, $n_3^\circ=|P^{(i)}|(3 \leq i \leq q-1)$ and $n_4^\circ=|P^{(q)}|$ in the following text.

Since $t \geq 3$, $l \geq 4$, and $n_1^\circ = (t-2)l+l-1 $, there exists two edges $w'w''=v_{(t-3)l+l-1}v_{(t-2)l}$, $ z'z''=v_{(t-2)l+\lfloor\frac{l-2}{2}\rfloor}v_{(t-2)l+\lfloor\frac{l-2}{2}\rfloor+1} \in E(P^{(1)})$.
Noting that $n-2 = \sum_{i=1}^{q} n_i \leq n_1 + (q-1)n_2$, which implies that $q \geq n_4^\circ + 4$ by $n \geq 10.2\times2^{(t-1)l+l-3}+2 \gg (t-1)l + l-1 + \lceil\frac{l-3}{2}\rceil^2 + 4\lfloor\frac{l-3}{2}\rfloor$.  
Then $P^{(3)}, P^{(4)}, \ldots, P^{(n_4^\circ+3)}$ are paths of order $n_3^\circ$.  
And assume that $P^{(q)} = w_1w_2 \cdots w_{n_4^\circ}$.
Let $G'$ be obtained from $G$ by 

(1) deleting $w'w''$ and joining $w''$ to an endpoint of $P^{(2)}$;

(2) deleting $z'z''$ and joining $z''$ to an endpoint of $P^{(3)}$;

(3) deleting all edges of $P^{(q)}$ and joining $w_i$ to an endpoint of $P^{(i+3)}$ for each $i \in \{1, \ldots, n_4^\circ\}$. 
Clearly, $G'$ is obtained from $G$ by deleting $n_4^\circ+1$ edges and adding $n_4^\circ + 2$ edges.  
By Lemma \ref{lm5}, then  
\[  
\frac{4}{\rho^2} \leq x_{u_i}x_{u_j} \leq \frac{4}{\rho^2} + \frac{24}{\rho^3} + \frac{36}{\rho^4} < \frac{4}{\rho^2} + \frac{25}{\rho^3}  
\]  
for any vertices $u_i, u_j \in R$. Therefore,  
\[  
\rho(G') - \rho \geq \frac{X^T(A(G') - A(G))X}{X^TX} > \frac{2}{X^TX} \left( \frac{4(n_4^\circ + 2)}{\rho^2} - \frac{4(n_4^\circ+1)}{\rho^2} - \frac{25(n_4^\circ+1)}{\rho^3} \right) > 0,  
\]  
where $n_4^\circ + 1 < \lfloor \frac{l-2}{2} \rfloor+1 \leq \frac{4}{25} \sqrt{2n-4} \leq \frac{4}{25} \rho$ as $n \geq \frac{625}{32} (\lfloor \frac{l-2}{2} \rfloor+1)^2 + 2$  and $\rho \geq \rho(K_{2,n-2})=\sqrt{2n-4}$ (since $K_{2,n-2}$ is a $tP_l$-free planar graph).
Yet, it is sad to say that $P_{n_1}$ in $G'$ has only $t-3$ independent $l$-paths, which means that $r = t-3$. Then, by Case \ref{case4.3},
we have $\rho = \rho(K_2 \vee H_{\mathcal{P}}((t-2)l+l-1,\lceil\frac{l-2}{2}\rceil,\lfloor\frac{l-2}{2}\rfloor) < \rho(K_2 \vee H^*) = \rho(K_2 \vee H_{\mathcal{P}}((t-3)l+l-1,l-1)$,  
which is impossible by the maximality of $\rho = \rho(G)$, leading to a contradiction.
Up to this case, we can conclude that $G\cong K_2 \vee H_{\mathcal{P}}((t-3)l+l-1,l-1)$.
\end{caseee}

\begin{casee}  
$r = t-1$ and $l \geq 6$. We denote that $P_{n_1} = v_1v_2 \cdots v_{n_1}$, $P' = v_1v_2 \cdots v_{(t-1)l}$ and $P_{\bar{n}_1} = v_{(t-1)l+1} \cdots v_{n_1}$. 
Thus, $K_2 \vee G[R]$ is $tP_l$ is equivalent to $K_2 \vee \bar{H}$ is $P_l$, where $K_2 \vee \bar{H} \cong K_2 \vee (G[R]-P')$.
Therefore, similar to the Case \ref{case4.2} of proof process of Theorem \ref{thm4} (ii), we have
 $\bar{n}_1+n_2+n_3 \leq l-3$ or $n_2+n_3+n_4 \leq l-3$, and $n_i=n_3$ for $i \in \{4,\ldots,q-1\}$.

For $l \geq 6$ and fixed two integers $1 \leq x \leq \lceil\frac{l-3}{3}\rceil$ and $1 \leq y \leq \lfloor\frac{l-3}{3}\rfloor$, we said $G[R] \cong H_{\mathcal{P}}((t-1)l+l-3-x-y,x,y)$.
If not,
we can obtain $G[R] \cong H_{\mathcal{P}}((t-1)l+l-3-x-y,x,y)$ by applying a series of $(s_1, s_2)$-transformations to $G[R]$, where $s_2 \leq s_1 \leq (t-1)l+l-1+l-1$.
As $n \geq 10.2 \times 2^{(t-1)(l-1)+l-3} + 2 \geq 10.2 \times 2^{s_2} + 2$, we conclude that $\rho < \rho(K_2 \vee H_{\mathcal{P}}((t-1)l+l-3-x-y,x,y)$ by Lemma \ref{lm5},  
which is impossible by the maximality of $\rho = \rho(G)$.
Now we assert that $\rho(K_2 \vee H_{\mathcal{P}}((t-1)l+l-3-x-y,x,y) < \rho(K_2 \vee H^*) = \rho(K_2 \vee H_{\mathcal{P}}((t-3)l+l-1,l-1)$.
Let $P^{(i)}$ to denote the $i$-th longest path of $H_{\mathcal{P}}((t-1)l+l-3-x-y,x,y)$ for any $i \in \{1, \ldots, q\}$ and $|P^{(i)}|$ to denote the orders of these $P^{(i)}$s, where $|P^{(1)}|=(t-1)l+l-3-x-y$, $|P^{(2)}|=x$, $|P^{(i)}|=y $ for $i \in \{3, \ldots, q-1\}$ and $|P^{(q)}| \geq 0$. For convenience, we write that $n_1^\circ=|P^{(1)}|$, $n_2^\circ=|P^{(2)}|$, $n_3^\circ=|P^{(i)}|(3 \leq i \leq q-1)$ and $n_4^\circ=|P^{(q)}|$ in the following text.

Since $t \geq 3$, $l \geq 6$, and $n_1^\circ = (t-1)l+l-3-x-y $, there are two edges $w'w''=v_{(t-3)l+l-1}v_{(t-2)l}$ and $z'z''=v_{(t-2)l+(l-1)-x-1}v_{(t-2)l+(l-1)-x} \in E(P^{(1)})$.
Noting that $n-2 = \sum_{i=1}^{q} n_i \leq n_1 + (q-1)n_2$, which implies that $q \geq n_4^\circ + 4$ by $n \geq 10.2\times2^{(t-1)l+l-3}+2 > (t-1)l + l-1 + \lceil\frac{l-3}{2}\rceil^2 + 4\lfloor\frac{l-3}{2}\rfloor$.  
Then $P^{(3)}, P^{(4)}, \ldots, P^{(n_4^\circ+3)}$ are paths of order $n_3^\circ$.  
And assume that $P^{(q)} = w_1w_2 \cdots w_{n_4^\circ}$.
Let $G'$ be obtained from $G$ by 

(1) deleting $w'w''$ and joining $w''$ to an endpoint of $P^{(2)}$;

(2) deleting $z'z''$ and joining $z''$ to an endpoint of $P^{(3)}$;

(3) deleting all edges of $P^{(q)}$ and joining $w_i$ to an endpoint of $P^{(i+3)}$ for each $i \in \{1, \ldots, n_4^\circ\}$. 
Clearly, $G'$ is obtained from $G$ by deleting $n_4^\circ+1$ edges and adding $n_4^\circ + 2$ edges.  
By Lemma \ref{lm5}, then  
\[  
\frac{4}{\rho^2} \leq x_{u_i}x_{u_j} \leq \frac{4}{\rho^2} + \frac{24}{\rho^3} + \frac{36}{\rho^4} < \frac{4}{\rho^2} + \frac{25}{\rho^3}  
\]  
for any vertices $u_i, u_j \in R$. Therefore,  
\[  
\rho(G') - \rho \geq \frac{X^T(A(G') - A(G))X}{X^TX} > \frac{2}{X^TX} \left( \frac{4(n_4^\circ + 2)}{\rho^2} - \frac{4(n_4^\circ+1)}{\rho^2} - \frac{25(n_4^\circ+1)}{\rho^3} \right) > 0,  
\]  
where $n_4^\circ + 1 <  y +1 \leq \lfloor \frac{l-3}{3} \rfloor+1 \leq \frac{4}{25} \sqrt{2n-4} \leq \frac{4}{25} \rho$ as $n \geq \frac{625}{32} (\lfloor \frac{l-3}{3} \rfloor+ 1 )^2 + 2$ and $\rho \geq \rho(K_{2,n-2})=\sqrt{2n-4}$ (since $K_{2,n-2}$ is a $tP_l$-free planar graph).
Yet, it is sad to say that $P_{n_1}$ in $G'$ has only $t-3$ independent $l$-paths, which means that $r = t-3$. Then, by Case \ref{case4.3},
we have $\rho=\rho(K_2 \vee H_{\mathcal{P}}((t-1)l+l-3-x-y,x,y) < \rho(K_2 \vee H_{\mathcal{P}}((t-3)l+l-1,l-1)) = \rho(K_2 \vee H^*)$,  
which is impossible by the maximality of $\rho = \rho(G)$, leading to a contradiction.  
\end{casee}



Now we considering that $G$ is disconnected. Let $n \geq N + \frac{3}{2}\sqrt{2N-6}$, where $N = \max\{2.67 \times 9^{17}, 10.2 \times 2^{s_2}+2\}$.
Suppose that $G$ is disconnected with components $G_1, G_2, \ldots, G_k$ ($k \geq 2$), and without loss of generality, assume $\rho(G_1) = \rho(G) = \rho$.
For each $i \in \{1, \ldots, k\}$, let $V(G_i) = \{u_{i,1}, u_{i,2}, \ldots, u_{i,n_i}\}$ and $t_i$ be the maximal number of independent $l$-paths of $G_i$.
Clearly, $\sum_{i=1}^k t_i \leq t-1$ and $\sum_{i=1}^k n_i = n$.
If $t_1 \leq t-2$, then constructing $G'$ by removing all the last edge of $P_l$s in $G_2, \ldots, G_k$, and then connecting just one of the endpoints of each $P_{l-1}$ in these components to $u_{1,1}$ in $G_1$ results in a connected $tP_l$-free planar graph.
Since $G_1$ is a proper subgraph of $G'$, $\rho(G') > \rho(G_1) = \rho(G)$, contradicting the maximality of $\rho(G)$. Thus, $t_1 = t-1$, it means that $G_2,G_3,\ldots,G_k$ is $P_l$-free.
Furthermore, $G_1$ must be an extremal graph in SPEX$_\mathcal{P}(n_1, tP_l)$, otherwise $G$ would not be an extremal in graph SPEX$_\mathcal{P}(n, tP_l)$.
If $n_1 \geq N$, then $G_1 \cong K_2 \vee H((t-3)l+l-1,l-1)$ for order $n_1$ by the case of connected graphs.
Now constructing a $G''$ by connecting all the longest paths in $G_2,G_3,\ldots,G_k$ to a common one of $K_2$ in $G_1$ results in a connected $tP_l$-free planar graph, again contradicting the maximality of $\rho(G)$.
Next, consider the case $n_1 < N$. By Lemma \ref{lm8}, we obtain
$$\rho=\rho(G_1)\leq 2+\sqrt{2n_1-6} < 2+\sqrt{2N-6} \leq \frac{1+\sqrt{8n-15}}{2}=\rho(K_2 \vee (n-2)K_1)$$
as $n\geq N+\frac{3}{2}\sqrt{2N-6}$ and $K_2 \vee (n-2)K_1$ is also $tP_l$-free for $t =1$ and $l \geq 6$ or $t \geq 2$ and $l \geq 4$ or $t \geq 3$ and $l \geq 3$.
This contradicts the maximality of $\rho(G)$.
Therefore, the extremal graph \(G\) must be connected. We then conclude that \(G\cong K_2\vee H_{\mathcal{P}}((t - 3)l + l - 1, l - 1)\), thus completing the proof. 
\end{proof}

\section{Concluding remarks}
In this paper, we have considered that the outerplanar and planar spectral Turán problems for linear forest and starlike tree with the same-length paths and determined the extremal graphs to \(\emph{spex}_\mathcal{P}(n,tP_l)\), \(\emph{spex}_\mathcal{OP}(n,tP_l)\), \(\emph{spex}_\mathcal{P}(n,P_{t\cdot l})\) and \(\emph{spex}_\mathcal{OP}(n,P_{t\cdot l})\) for sufficiently large \( n \), respectively.

We note that determining the final extremal graph to $\emph{spex}_\mathcal{OP}(n,P_{t\cdot l})$ and $\emph{spex}_\mathcal{P} (n,P_{t\cdot l})$ when $t = 3$, and when $t=3,4$ for sufficiently large $n$ involves considering numerous complex situations (e.g. the central vertex $v_1$ may be in $G[A]$ (or $G[R]$)). Although our specific proofs only establish that the extremal graph structure as $G\cong K_{1} \vee G[A]$ (or $G\cong K_{2} \vee G[R]$), where $G[A]$ (or $G[R]$) is a disjoint union of paths, we believe that applying the same analytical approaches and methods used in this paper can fully characterize the final extremal graphs. Based on this reasoning, we put forward the following problems about these extremal graphs:

\begin{problem}
When $t = 3$, $l \geq 3$ and $n$ is large enough, is the extremal graph to $\textit{spex}_{\mathcal{OP}}(n,P_{t\cdot l})$  $K_1 \vee H_{\mathcal{OP}}(2l - 2, l - 2)$?
\end{problem}

\begin{problem}
When $t = 3$, $l \geq 3$ and $n$ is large enough, is the extremal graph to $\textit{spex}_{\mathcal{P}}(n,P_{t\cdot l})$ $K_2 \vee H_{\mathcal{P}}(l - 1, l - 2)$?
\end{problem}

\begin{problem}
When $t = 4$, $l \geq 3$ and $n$ is large enough, is the extremal graph to $\textit{spex}_{\mathcal{P}}(n,P_{t\cdot l})$ $K_2 \vee H_{\mathcal{P}}(2l - 2, l - 2)$?
\end{problem}

For general linear forest $F(l_1,l_2,\cdots,l_t)$ and starlike tree $S(l_1,l_2,\cdots,l_t)$, two more important and interesting problems are also put forward as follows.

\begin{problem}
  Characterize planar and outerplanar spectral extremal graphs without linear forest $F(l_1,l_2,\cdots,l_t)$ for sufficiently large $n$-vertex.
\end{problem}

\begin{problem}
  Characterize planar and outerplanar spectral extremal graphs without starlike tree $S(l_1,l_2,\cdots,l_t)$ for sufficiently large $n$-vertex.
\end{problem}

\section{Acknowledgement}
We would like to show our great gratitude to anonymous referees for their valuable suggestions which greatly improved the quality of this paper.

\end{document}